\documentclass[11pt]{amsart}
\usepackage{amsmath, amsfonts, amssymb, amsthm}
\usepackage{enumerate}
\usepackage[inline]{enumitem}
\usepackage{graphicx}
\usepackage{xcolor}
\usepackage{soul}
\usepackage[colorlinks,citecolor=blue]{hyperref}
\usepackage{verbatim}
\usepackage[headings]{fullpage}

\newtheorem{theorem}{Theorem}[section]
\newtheorem{lemma}[theorem]{Lemma}
\newtheorem{proposition}[theorem]{Proposition}
\newtheorem{claim}[theorem]{Claim}
\newtheorem{cor}[theorem]{Corollary}

\theoremstyle{definition} 
\newtheorem{definition}[theorem]{Definition}
\newtheorem{example}[theorem]{Example}

\newtheorem{conjecture}[theorem]{Conjecture}
\newtheorem{remark}[theorem]{Remark}

\numberwithin{equation}{section}

\DeclareMathOperator{\Stab}{Stab}%
\DeclareMathOperator{\Isom}{Isom}%
\DeclareMathOperator{\Out}{Out}%
\DeclareMathOperator{\Aut}{Aut}%
\DeclareMathOperator{\BBT}{BBT}%
\DeclareMathOperator{\inter}{int}%
\DeclareMathOperator{\rk}{rk}%
\newcommand{\rank}{\mathfrak{n}}%
\newcommand{\alert}{\textcolor{red}}%
\newcommand{\la}{\langle}%
\newcommand{\ra}{\rangle}%
\newcommand{\CyclicS}{\mathcal{FZ}}%
\newcommand{\FreeS}{\mathcal{FS}}%
\newcommand{\m}{\text{(max)}}%
\newcommand{\FreeF}{\mathcal{FF}}%
\newcommand{\ffa}{\mathcal{A}}%
\newcommand{\free}{\mathbb{F}}%
\newcommand{\cv}{\mathcal{O}}%
\newcommand{\cvclo}{\overline{\mathcal{O}}}%
\newcommand{\CV}{\mathbb{P}\mathcal{O}}%
\newcommand{\CVclo}{\overline{\mathbb{P} \mathcal{O}}}%
\newcommand{\CVbound}{\partial \mathbb{P} \mathcal{O}}%
\newcommand{\relcv}{\mathcal{O}(\free, \mathcal{A})}%
\newcommand{\relCVClo}{\overline{\mathbb{P}\mathcal{O}(\free,
    \mathcal{A})}} 

\newcommand{\lamination}{\Lambda_{\phi}^+} 
\newcommand{\Rlamination}{\Lambda^-_{\phi}} 
\newcommand{\StableTree}{T^+_{\phi}}
\newcommand{\UnstableTree}{T^-_{\phi}}
\newcommand{\PMTrees}{T^{\pm}_{\phi}}
\newcommand{\PMlaminations}{\Lambda^{\pm}_{\phi}}
\newcommand{\MPlaminations}{\Lambda^{\mp}_{\phi}}
\newcommand{\PFevalue}{\lambda_{\phi}}
\newcommand{\Rfltr}{\emptyset = G_0 \subset G_1 \subset \cdots \subset
  G_r = G} %
\newcommand{\fltr}{\emptyset = G_0 \subset G_1 \subset
  \cdots \subset G_M = G}

\title{Loxodromic elements in the cyclic splitting complex and their
  centralizers}
\author{Radhika Gupta}
\address{\tt R.\ Gupta, Department of Mathematics, Technion, \,
  \newline Haifa, Israel, 32000
  \newline http://www.math.utah.edu/\~{}gupta/} %
\email{\tt radhikagup@technion.ac.il}
\author{Derrick Wigglesworth} %
\address{\tt D.\ Wigglesworth, Department of Mathematical Sciences, University
  of Arkansas, 309 SCEN, Fayetteville, AR 72703, U.S.A.
  \newline http://www.math.utah.edu/$\sim$dwiggles/} %
\email{\tt drwiggle@uark.edu}
\thanks{\today}

\begin{document}

\begin{abstract}
  We show that an outer automorphism acts loxodromically on the cyclic
  splitting complex if and only if it has a filling lamination and no
  generic leaf of the lamination is carried by a vertex group of a
  cyclic splitting.  This is the analog for the cyclic splitting
  complex of Handel-Mosher's theorem on loxodromics for the free
  splitting
  complex. 
  We also show that such outer automorphisms have virtually cyclic
  centralizers.
\end{abstract}

\thanks{ Both authors are partially supported by the U.S.\ National
  Science Foundation grant of Mladen Bestvina (DMS-1607236).}


\maketitle
\section{Introduction}
\label{sec:intro}

The study of the mapping class group of a closed orientable surface
$S$ has benefited greatly from its action on the curve complex,
$\mathcal{C}(S)$, which was shown to be hyperbolic in
\cite{MM:CurveComplex}.  Curve complexes have been used for bounded
cohomology of subgroups of mapping class groups, rigidity results, and
myriad other applications.

The outer automorphism group of a finite rank free group $\free$,
denoted by $\Out(\free)$, is defined as the quotient of $\Aut(\free)$
by the inner automorphisms, those which arise from conjugation by a
fixed element.  Much of the study of $\Out(\free)$ draws parallels
with the study of mapping class groups.  This analogy, however, is far
from perfect; there are several $\Out(\free)$-complexes that act as
analogs for the curve complex.  Among them are the free splitting
complex $\FreeS$, the cyclic splitting complex $\CyclicS$, and the
free factor complex $\FreeF$, all of which have been shown to be
hyperbolic \cite{HM:FreeSplittingComplex, M:CyclicS,
  BF:FreeFactorComplex}.  Just as curve complexes have yielded useful
information about mapping class groups, so too have these complexes
furthered our understanding of $\Out(\free)$.

The three hyperbolic $\Out(\free)$-complexes mentioned above are
related via coarse Lipschitz maps, $\FreeS \to \CyclicS \to \FreeF$. The
loxodromics for $\FreeF$ have been identified with the set of fully
irreducible outer automorphisms \cite{BF:FreeFactorComplex}.  In
\cite{HM:FreeSplittingComplexII}, the authors proved that an outer
automorphism, $\phi$, acts loxodromically on $\FreeS$ precisely when
$\phi$ has a \emph{filling lamination}, that is, some element of the
finite set of laminations associated to $\phi$ (see \cite{BFH:Tits})
is not carried by a vertex group of any free splitting.  In this
paper, we focus our attention on the isometry type of outer
automorphisms, considered as elements of $\Isom(\CyclicS)$.

A $\mathcal{Z}$-splitting of $\free$ is a splitting in which edge
stabilizers are either trivial or cyclic.  The cyclic splitting
complex $\CyclicS$, introduced in \cite{M:CyclicS}, is defined as
follows (see Section~\ref{subsec:CyclicS}): vertices are one-edge
$\mathcal{Z}$-splittings of $\free$ and $k$-simplicies correspond to
collections of $k+1$ vertices which are compatible with a common
$k+1$-edge $\mathcal{Z}$-splitting.  In this paper, we determine
precisely which outer automorphisms act loxodromically on $\CyclicS$.
Closely related to $\mathcal{Z}$-splittings are the maximally-cyclic
splittings, called $\mathcal{Z}^{\text{max}}$-splittings, in which the
edge groups are required to be trivial or maximal cyclic (i.e., not
contained in a larger cyclic subgroup).  The results of this paper
also apply to the maximally-cyclic splitting complex
$\CyclicS^{\text{max}}$ which is defined exactly as $\CyclicS$ except
that splittings are required to be in the class
$\mathcal{Z}^{\text{max}}$.  We will use the notation $\CyclicS^{\m}$
to mean either $\CyclicS$ or $\CyclicS^{\text{max}}$.

In \cite{BFH:Tits}, the authors associate to each $\phi\in\Out(\free)$
a finite set of attracting laminations, denoted by
$\mathcal{L}(\phi)$.  We say that a lamination
$\Lambda\in\mathcal{L}(\phi)$ is \emph{$\mathcal{Z}^{\m}$-filling} if
no generic leaf (see \S\ref{sec:lam} for definitions) of $\Lambda$ is
carried by a vertex group of a one-edge $\mathcal{Z}^{\m}$-splitting;
we say that $\phi$ has a $\mathcal{Z}^{\m}$-filling lamination if some
element of $\mathcal{L}(\phi)$ is $\mathcal{Z}^{\m}$-filling.  We
prove

\begin{theorem}\label{thm:loxodromics}
  For a free group of rank at least 3, an outer automorphism $\phi$
  acts loxodromically on $\CyclicS^{\m}$ if and only if it has a
  $\mathcal{Z}^{\m}$-filling lamination.  Furthermore, if $\phi$ has a
  filling lamination which is not $\mathcal{Z}^{\m}$-filling, then a
  power of $\phi$ fixes a point in $\CyclicS^{\m}$.
\end{theorem}

In \cite{HW:Aut}, Horbez and Wade showed that every isometry of
$\CyclicS^{\m}$ is induced by an outer automorphism.  Combining their
result with \cite[Theorem 1.1]{HM:FreeSplittingComplexII} and Theorem
\ref{thm:loxodromics}, this amounts to a classification of the
isometries of $\CyclicS^{\m}$.

\begin{cor}[Classification of isometries]\label{thm:classification}
  The following hold for all $\phi\in\Isom(\CyclicS^{\m})$.
  \begin{enumerate}
  \item The action of $\phi$ on $\CyclicS^{\m}$ is loxodromic if and
    only if some element of $\mathcal{L}(\phi)$ is
    $\mathcal{Z}^{\m}$-filling.
  \item If the action of $\phi$ on $\CyclicS^{\m}$ is not loxodromic,
    then it has bounded orbits (there are no parabolic isometries).
  \end{enumerate}
\end{cor}

The proof of Theorem~\ref{thm:loxodromics} relies on the description
of the boundary of $\CyclicS^{\m}$ due to Horbez
\cite{H:CyclicSBoundary}; points in the boundary of $\CyclicS^{\m}$
are equivalence classes of $\mathcal{Z}^{\m}$-averse trees.  The proof
is carried out as follows.  In Section~\ref{sec:folding-seq}, we
extend the theory of folding paths to the boundary of Culler \&
Vogtmann's outer space, $\CV$, defining a folding path guided by
$\phi$ which is entirely contained in $\CVbound$. In
Section~\ref{sec:mixing}, we show that the limit of the folding path
thus constructed is $\mathcal{Z}^{\m}$-averse. In
Section~\ref{sec:filling not Z-filling}, we show that an outer
automorphism with a filling but not $\mathcal{Z}^{\m}$-filling
lamination fixes (up to taking a power) a point in $\CyclicS^{\m}$ and
conclude with a proof of Theorem~\ref{thm:loxodromics}.

The remainder of the paper is devoted to a study of the centralizers
of automorphisms with filling laminations.  We prove the following
result:

\begin{theorem}\label{thm:centralizers}
  If an outer automorphism $\phi$ has a $\mathcal{Z}$-filling
  lamination, then its centralizer in $\Out(\free)$ is virtually
  cyclic.  Conversely, if $\phi$ has a filling but not
  $\mathcal{Z}$-filling lamination, then the centralizer of some power
  of $\phi$ in $\Out(\free)$ is not virtually cyclic.
\end{theorem}

The key tools used to prove Theorem~\ref{thm:centralizers} are the
completely split train tracks introduced in
\cite{FH:RecognitionTheorem} and the disintegration theory for outer
automorphisms developed in \cite{FH:Abelian}.  We first show
(Proposition~\ref{prop:D-cyclic}) that the disintegration of any outer
automorphism $\phi$, that has a $\mathcal{Z}$-filling lamination, is
virtually cyclic. Then we show that Proposition~\ref{prop:D-cyclic}
implies the centralizer of $\phi$ is also virtually cyclic.
Conversely, in Proposition~\ref{prop:converse}, we show that if $\phi$
has a filling lamination that is not $\mathcal{Z}$-filling, then
$\phi$ commutes with an appropriately chosen partial conjugation.

The method used to prove Theorem~\ref{thm:centralizers} provides
alternate (and simple) proof of the well-known fact due to Bestvina,
Feighn and Handel that centralizers of fully irreducible outer
automorphisms are virtually cyclic.  In \cite{BFH:Tits}, the stretch
factor homomorphism is used to show that the stabilizer of the
lamination of a fully irreducible outer automorphism is virtually
cyclic, which implies that the centralizer is also virtually cyclic.
In general, little is known about the centralizers of outer
automorphisms.  In \cite{RW:DehnTwist}, Rodenhausen and Wade describe
an algorithm to find a presentation of the centralizer of a Dehn Twist
automorphism.  In \cite{FH:Abelian}, Feighn and Handel show that the
disintegration of an outer automorphism $\mathcal{D}(\phi)$ is
contained in the weak center of the centralizer of $\phi$.  Recently,
Algom-Kfir and Pfaff showed \cite{AP:NormalizerCentralizer} that
centralizers of fully irreducible outer automorphisms with lone axes
are isomorphic to $\mathbb{Z}$.  We also mention a result of Kapovich
and Lustig \cite{KL:StabTree}: automorphisms whose limiting trees are
free have virtually cyclic centralizers.

The main motivation for examining the centralizers of loxodromic
elements of $\CyclicS$ (and $\FreeS$) is to understand which
automorphisms have the potential to be WPD elements for the action of
$\Out(\free)$ on these complexes.

\begin{cor}
  Any outer automorphism that is loxodromic for the action of
  $\Out(\free)$ on $\FreeS$ but elliptic for the action on $\CyclicS$
  is not a WPD element for the action on $\FreeS$.
\end{cor}

The result that centralizers of loxodromic elements of $\CyclicS$ are
virtually cyclic is a promising sign for the following conjecture:

\begin{conjecture}
  The action of $\Out(\free)$ on $\CyclicS$ is a WPD action.  That is,
  every loxodromic element for the action satisfies WPD.
\end{conjecture}

\subsection*{Acknowledgements:} We would like to thank Mladen Bestvina
for suggesting this project, sharing his intuition and for the many
discussions we had with him.  We are especially grateful to Carolyn
Abbott for her sustained enthusiasm and interest this project.  We
would also like to thank the referee for many helpful suggestions.
The first author would also like to thank Camille Horbez for his
encouragement in the very early stage of this project.  We also
acknowledge Mathematical Sciences Research Institute (MSRI), Berkeley,
California, where we first became interested in this project.

\section{Preliminaries}
Before proceeding, we fix a free group $\free$ of rank $\geq 3$.

\subsection{Isometries of metric spaces}
Let $X$ be a Gromov hyperbolic metric space.  We say that an infinite
order isometry $g$ of $X$ is \emph{loxodromic} if it acts with
positive translation length on $X$:
$\lim_{N \to \infty} \frac{d(x, g^N(x))}{N}>0$ for some (any)
$x \in X$.  Every loxodromic element has exactly two limit points in
the Gromov boundary of $X$.

Given a group $G$ acting by isometries on the hyperbolic space $X$, we
denote by $\Lambda_X G$ the limit set of $G$ in $\partial_\infty X$,
which is defined as the intersection of $\partial_\infty X$ with the
closure of the orbit of any point in $X$ under the $G$-action.  The
following theorem, essentially due to Gromov, and formulated here for
the case that $G$ is cyclic, gives a classification of isometry groups
of (possibly nonproper) Gromov hyperbolic spaces. A sketch of proof
can be found in \cite[Proposition 3.1]{CCM:AmenableHyperbolicGps}.
\begin{theorem}[{\cite[Section 8.2]{Gromov:HypGps}}]\label{thm:hyp-isom}
  Let $X$ be a hyperbolic geodesic metric space, and let $G$ be a
  cyclic group acting by isometries on $X$. Then $G$ is either
  \begin{itemize}
  \item bounded, i.e.\ all $G$-orbits in $X$ are bounded; in this case
    $\Lambda_X G = \emptyset$, or
  \item horocyclic, i.e.\ $G$ is not bounded and contains no
    loxodromic element; in this case $\Lambda_X G$ is reduced to one
    point, or
  \item lineal, i.e.\ $G$ contains a loxodromic element, and any two
    loxodromic elements have the same fixed points in
    $\partial_\infty X$; in this case $\Lambda_X G$ consists of these
    two points.
  \end{itemize}
\end{theorem}

\subsection{Outer space and its compactification}
Culler Vogtmann's \emph{outer space}, $\CV$, is defined in
\cite{CV:OuterSpace} as the space of simplicial, free, and minimal
isometric actions of $\free$ on simplicial metric trees up to
$\free$-equivariant homothety.  We denote by $\cv$ the
\emph{unprojectivized outer space}, in which the trees are considered
up to isometry, rather than homothety.  Each of these spaces is
equipped with a natural (right) action of $\Out(\free)$.

An $\free$-tree is an $\mathbb{R}$-tree with an isometric action of
$\free$. An $\free$-tree is called \emph{very small} if the action is
minimal, arc stabilizers are either trivial or maximal cyclic, and
tripod stabilizers are trivial. Outer space can be mapped into
$\mathbb{R}^{\free}$ by the map $T\mapsto (\|g\|_T)_{g\in\free}$,
where $\|g\|_T$ denotes the translation length of $g$ in $T$.  This
was shown in \cite{CM:LengthFunction} to be a continuous injection.
The closure of the image of $\CV$ under this embedding into is compact
and was identified in \cite{BF:OuterLimits} and
\cite{CL:BoundaryOuterSpace} with the space of very small
$\free$-trees.  We denote by $\CVclo$ the closure of outer space in
$\mathbb{PR}^{\free}$ and by $\CVbound$ its boundary.  We will denote
the preimage of $\CVclo$ in $\mathbb{R}^{\free}$ by $\cvclo$.

\subsection{Free factor system}
A free factor system of $\free$ is a finite collection of conjugacy
classes of proper free factors of $\free$ of the form
$\ffa = \{[A_1], \ldots, [A_k] \}$, where $k\geq 0$ and $[\cdot]$
denotes the conjugacy class of a subgroup, such that there exists a
free factorization $\free = A_1 \ast \cdots \ast A_k \ast F_N$. We
refer to the free factor $F_N$ as the \emph{cofactor} of $\ffa$
keeping in mind that it is not unique, even up to conjugacy.
   
The main geometric example of a free factor system is as follows:
suppose $G$ is a marked graph and $K$ is a subgraph whose
non-contractible connected components are denoted $C_1, \ldots,
C_k$. Let $[A_i]$ be the conjugacy class of a free factor of $\free$
determined by $\pi_1(C_i)$. Then $\ffa =\{ [A_1], \ldots, [A_k]\}$ is
a free factor system. We say $\ffa$ is \emph{realized by $K$} and we
denote it by $\mathcal{F}(K)$.

\subsection{Marked graphs}\label{sec:marked-graphs}
We recall some basic definitions from \cite{BH:TrainTracks}. Identify
$\free$ with $\pi_1(\mathcal{R}, \ast)$ where $\mathcal{R}$ is a rose
with $\rank$ petals, $\rank$ being the rank of $\free$. A \emph{marked
  graph} $G$ is a graph of rank $\rank$, all of whose vertices have
valence at least three, equipped with a homotopy equivalence
$m\colon \mathcal{R}\to G$ called a \emph{marking}. The marking
determines an identification of $\free$ with $\pi_1(G,m(\ast))$. A
homotopy equivalence $f\colon G \to G$ induces an outer automorphism
of $\pi_1(G)$ and hence an element $\phi$ of $\Out(\free)$. If $f$
sends vertices to vertices and the restriction of $f$ to edges is an
immersion then we say that $f$ is a \emph{topological representative}
of $\phi$.

\subsection{Paths, circuits, and tightening}
Let $\Gamma$ be either a marked graph or an $\free$-tree.  A
\emph{path} in $\Gamma$ is either an isometric immersion of a
(possibly infinite) closed interval $\sigma\colon I\to \Gamma$ or a
constant map $\sigma\colon I\to \Gamma$.  If $\sigma$ is a constant
map, the path will be called \emph{trivial}.  If $I$ is finite, then
any map $\sigma\colon I\to \Gamma$ is homotopic rel endpoints to a
unique path $[\sigma]$.  We say that $[\sigma]$ is obtained by
\emph{tightening} $\sigma$.  If $f\colon \Gamma\to \Gamma$ is
continuous and $\sigma$ is a path in $\Gamma$, we define
$f_\#(\sigma)$ as $[f(\sigma)]$.  If the domain of $\sigma$ is finite
and $\Gamma$ is either a graph or a simplicial tree, then the image
has a natural decomposition into edges $E_1E_2\cdots E_k$ called the
\emph{edge path associated to} $\sigma$.  If $\Gamma$ is a tree, we
may use $[x,x']$ to denote the unique geodesic path connecting $x$ and
$x'$.

A \emph{circuit} is an immersion $\sigma\colon S^1\to \Gamma$.  For
any path or circuit, let $\overline{\sigma}$ be $\sigma$ with its
orientation reversed.  A decomposition of a path or circuit into
subpaths is a \emph{splitting} for $f\colon \Gamma\to \Gamma$ and is
denoted $\sigma=\ldots\sigma_1\cdot\sigma_2\ldots$ if
$f^k_\#(\sigma)=\ldots f^k_\#(\sigma_1)f^k_\#(\sigma_2)\ldots$ for all
$k\geq 1$.

\subsection{Turns, directions and train track structure}
Let $\Gamma$ be an $\free$-tree.  A direction $d$ based at $p\in\Gamma$
is a component of $\Gamma -\{p\}$.  A \emph{turn} is an unordered pair
of directions based at the same point.  
In the case that $\Gamma$ is a simplicial tree, and $p$ is a vertex, we identify
directions at $p$ with edges emanating from $p$.  An \emph{illegal
  turn structure} on $\Gamma$ is an equivalence relation on the set of
directions at each point $p\in\Gamma$.  The classes of this relation
are called \emph{gates}.  A turn $(d,d')$ is \emph{legal} if $d$ and
$d'$ do not belong to the same gate.  If in addition there are at
least two gates at every vertex of $\Gamma$, then the illegal turn
structure is called a \emph{train track structure}.  A path is legal
if it only crosses legal turns.

\subsection{Optimal morphism}
Given two $\free$-trees $\Gamma$ and $\Gamma'$, an $\free$-equivariant
map $f\colon \Gamma \to \Gamma'$ is called a \emph{morphism} if every
segment of $\Gamma$ can be subdivided into finitely many subintervals
onto which $f$ restricts to an isometric embedding.  A morphism
between $\free$-trees induces an illegal turn structure structure on
the domain $\Gamma$ as follows: for every $x \in \Gamma$, the map $f$
determines a map $Df_x \colon D_x \to D_{f(x)}$, on the set of
directions $D_x$ at $x$.  For $d, d' \in D_x$, we then declare
$d\sim d'$ if $D(f^k)(d)=D(f^k)(d')$ for some $k\geq 0$.  A morphism
is called \emph{optimal} if there are at least two gates at each point
of $\Gamma$.  A morphism $f$ that induces a train track structure is
an optimal
morphism. 

The map $f$ is called \emph{alignment preserving}
(or a \emph{collapse map}) if the $f$-image of every segment in
$\Gamma$ is a segment in $\Gamma'$.


\subsection{Train track maps}

An optimal morphism is called a \emph{train track map} if
$f \colon \Gamma \to \Gamma'$ is an embedding on each edge and maps legal
turns to legal turns.  In particular, legal paths map to legal
paths.  Note that usually the term \emph{train track map} is used for
self maps, but in \cite{BF:FreeFactorComplex}, the authors define it
for a map between different $\free$-trees, each equipped with its own
abstract train track structure.

The terminology can also be exteded to graphs by passing to their
universal covers.  For more details on train track maps, the reader is
referred to \cite{BF:FreeFactorComplex,BH:TrainTracks}.

\subsection{Relative train track maps and CTs}
A \emph{filtration} for a topological representative $f\colon G\to G$
of an outer automorphism $\phi$, where $G$ is a marked graph, is an
increasing sequence of $f$-invariant subgraphs $\fltr$.  We let
$H_i=\overline{G_i\setminus G_{i-1}}$ and call $H_i$ the \emph{$i$-th
  stratum}.  A turn with one edge in $H_i$ and the other in $G_{i-1}$
is called \emph{mixed} while a turn with both edges in $H_i$ is called
a \emph{turn in $H_i$}.  If $\sigma\subset G_i$ does not contain any
illegal turns in $H_i$, then we say $\sigma$ is \emph{$i$-legal}.

We denote by $M_i$ the submatrix of the transition matrix for $f$
obtained by deleting all rows and columns except those labeled by
edges in $H_i$.  For the topological representatives that will be of
interest to us, the transition matrices $M_i$ will come in three
flavors: $M_i$ may be a zero matrix, it may be the $1\times 1$
identity matrix, or it may be an irreducible matrix with
Perron-Frobenius eigenvalue $\lambda_i>1$.  We will call $H_i$ a
\emph{zero} (Z), \emph{non-exponentially growing} (NEG), or
\emph{exponentially growing} (EG) stratum, respectively.  Any stratum
which is not a zero stratum is called an \emph{irreducible stratum}.

\begin{definition}[{\cite{BH:TrainTracks}}]
  We say that $f\colon G\to G$ is a \emph{relative train track map}
  representing $\phi\in\Out(F_n)$ if for every exponentially growing
  stratum $H_r$, the following hold:
  \begin{description}
  \item[(RTT-i)]\label{RTT-i} $Df$ maps the set of oriented edges in
    $H_r$ to itself; in particular all mixed turns are legal.
  \item[(RTT-ii)]\label{RTT-ii} If $\sigma\subset G_{r-1}$ is a
    nontrivial path with endpoints in $H_r\cap G_{r-1}$, then so is
    $f_\#(\sigma)$.
  \item[(RTT-iii)]\label{RTT-iii} If $\sigma\subset G_r$ is $r$-legal,
    then $f_\#(\sigma)$ is $r$-legal.
  \end{description}
\end{definition}

Suppose that $u<r$, that $H_u$ is irreducible, $H_r$ is EG and each
component of $G_r$ is non-contractible, and that for each $u<i<r$,
$H_i$ is a zero stratum which is a component of $G_{r-1}$ and each
vertex of $H_i$ has valence at least two in $G_r$.  Then we say that
$H_i$ is \emph{enveloped by $H_r$} and we define
$H_r^z=\bigcup_{k=u+1}^rH_k$.

A path or circuit $\sigma$ in a representative $f\colon G\to G$ is
called a \emph{periodic Nielsen path} if $f_\#^k(\sigma)=\sigma$ for
some $k\geq 1$.  If $k=1$, then $\sigma$ is a \emph{Nielsen path}.  A
Nielsen path is \emph{indivisible}, denoted INP, if it cannot be
written as a concatenation of non-trivial Nielsen paths.  If $w$ is a
closed root-free Nielsen path and $E_i$ is an edge such that
$f(E_i)=E_iw^{d_i}$, then we say \emph{$E_i$ is a linear edge} and we
call $w$ the \emph{axis} of $E$.  If $E_i,E_j$ are distinct linear
edges with the same axis such that $d_i\neq d_j$ and $d_i,d_j>0$, then
we call a path of the form $E_iw^*\overline{E}_j$ an \emph{exceptional
  path}.  We say that $x$ and $y$ are \emph{Nielsen equivalent} if
there is a Nielsen path $\sigma$ in $G$ whose endpoints are $x$ and
$y$.  We say that a periodic point $x\in G$ is \emph{principal} if
neither of the following conditions hold:
\begin{itemize}
\item $x$ is an endpoint of a non-trivial periodic Nielsen path and
  there are exactly two periodic directions at $x$, both of which are
  contained in the same EG stratum.
\item $x$ is contained in a component $C$ of periodic points that is
  topologically a circle and each point in $C$ has exactly two
  periodic directions.
\end{itemize}

A relative train track map $f$ is called \emph{rotationless} if each
principal periodic vertex is fixed and if each periodic direction
based at a principal vertex is fixed.

For an EG stratum, $H_r$, we call a non-trivial path
$\sigma\subset G_{r-1}$ with endpoints in $H_r\cap G_{r-1}$ a
\emph{connecting path for $H_r$}.  Let $E$ be an edge in an
irreducible stratum, $H_r$ and let $\sigma$ be a maximal subpath of
$f_\#^k(E)$ in a zero stratum for some $k\geq 1$.  Then we say that
$\sigma$ is \emph{taken}.  A non-trivial path or circuit $\sigma$ is
called \emph{completely split} if it has a splitting
$\sigma=\tau_1\cdot \tau_2\cdots\tau_k$ where each of the $\tau_i$'s
is a single edge in an irreducible stratum, an indivisible Nielsen
path, an exceptional path, or a connecting path in a zero stratum
which is both maximal and taken.  We say that a relative train track
map is \emph{completely split} if $f(E)$ is completely split for every
edge $E$ in an irreducible stratum \emph{and} if for every taken
connecting path $\sigma$ in a zero stratum, $f_\#(\sigma)$ is
completely split.

The following theorem/definition is the main existence result for CTs:

\begin{theorem}[{\cite[Theorem
    4.28]{FH:RecognitionTheorem}}{\cite[Corollary
    3.5]{FH:Abelian}}] \label{def:ct} There exists $k>0$ depending
  only on $\rank$, so that given any $\phi\in\Out(\free)$ and any
  nested sequence of $\phi^k$-invariant free factor systems, there is
  a \emph{completely split improved relative train track map}
  (\emph{CT} for short) $f\colon G\to G$ representing $\phi^k$ such
  that each free factor system is realized by some filtration element.
  The map $f$ satisfies the following properties:
  \begin{description}
  \item[(Rotationless)] $f\colon G\to G$ is rotationless.
  \item[(Completely Split)] $f\colon G\to G$ is completely split.
  \item[(Filtration)] $\mathcal{F}$ is reduced.  The core of each
    filtration element is a filtration element.
  \item[(Vertices)] The endpoints of all indivisible periodic
    (necessarily fixed) Nielsen paths are (necessarily principal)
    vertices.  The terminal endpoint of each non-fixed NEG edge is
    principal (and hence fixed).
  \item[(Periodic Edges)] Each periodic edge is fixed and each
    endpoint of a fixed edge is principal.  If the unique edge $E_r$
    in a fixed stratum $H_r$ is not a loop then $G_{r-1}$ is a core
    graph and both ends of $E_r$ are contained in $G_{r-1}$.
  \item[(Zero Strata)] If $H_i$ is a zero stratum, then $H_i$ is
    enveloped by an EG stratum $H_r$, each edge in $H_i$ is $r$-taken
    and each vertex in $H_i$ is contained in $H_r$ and has link
    contained in $H_i \cup H_r$.
  \item[(Linear Edges)] For each linear $E_i$ there is a closed
    root-free Nielsen path $w_i$ such that $f(E_i) = E_i w_i^{d_i}$
    for some $d_i \ne 0$.  If $E_i$ and $E_j$ are distinct linear
    edges with the same axes then $w_i = w_j$ and $d_i \ne d_j$.
  \item[(NEG Nielsen Paths)] If the highest edges in an indivisible
    Nielsen path $\sigma$ belong to an NEG stratum then there is a
    linear edge $E_i$ with $w_i$ as in (Linear Edges) and there exists
    $k \ne 0$ such that $\sigma = E_i w_i^k \bar E_i$.
  \end{description}
  Moreover, if $\phi$ is rotationless in the sense of
  \cite{FH:RecognitionTheorem}, then we may take $k=1$.
\end{theorem}

It follows directly from the definitions that, for completely split
paths and circuits, all cancellation under iteration of $f_{\#}$ is
confined to the individual terms of the splitting.  Moreover,
$f_\#(\sigma)$ has a complete splitting which refines that of
$\sigma$.  Finally, just as with improved relative train track maps
introduced in \cite{BFH:Tits}, every circuit or path with endpoints at
vertices eventually is completely split \cite[Lemma
4.25]{FH:RecognitionTheorem}.  The reader is directed to \cite[\S
4]{FH:RecognitionTheorem} for many useful properties of CTs that we
will use frequently in the sequel, often without a specific reference.

\subsection{Bounded backtracking (BBT)} Let $f \colon T \to T'$ be a
continuous map between two $\mathbb{R}$-trees $T$ and $T'$. We say
that $f$ has bounded backtracking if the $f$ image of any path $[p,q]$
is contained in a $C$-neighborhood of $[f(p),f(q)]$. The smallest such
$C$ is called the \emph{bounded backtracking constant} of $f$, denoted
$\BBT(f)$.

\subsection{Folding paths}
Given simplicial $\free$-trees $T$ and $T'$ and an optimal morphism
$f\colon T\to T'$ Guirardel and Levitt \cite[Section
3]{GL:RelativeOuterSpace} construct a \emph{canonical optimal folding
  path} $(T_t)_{t\in\mathbb{R}^+}$ guided by $f$.  The tree $T_t$ is
constructed as follows.  Given $a,b\in T$ with $f(a)=f(b)$, the
\emph{identification time} of $a$ and $b$ is defined as
$\tau(a,b)=\sup_{x\in[a,b]}d_{T'}(f(x),f(a))$.  Define
$L:=\frac{1}{2}\BBT(f)$.  For each $t\in [0,L]$, one defines an
equivalence relation $\sim_t$ by $a\sim_t b$ if $f(a)=f(b)$ and
$\tau(a,b)<t$.  The tree $T_t$ is then a quotient of $T$ by the
equivalence relation $\sim_t$.  The authors prove that for each
$t\in[0,L]$, $T_t$ is an $\mathbb{R}$-tree.  The collection of trees
$(T_t)_{t\in[0,L]}$ comes equipped with $\free$-equivariant morphisms
$f_{s,t}\colon T_t\to T_s$ for all $t<s$ and these maps satisfy the
semi-flow property: for all $r<s<t$, we have
$f_{t,s}\circ f_{s,r}=f_{t,r}$.  Moreover $T_L=T'$ and $f_{L,0}=f$.
The trees $(T_t)_{t \in [0,L]}$ and the maps
$(f_{s,t}\colon T_t \to T_s)_{t<s \in [0,L]}$ are called the
\emph{connection data} for the folding path.

\subsection{The $\mathcal{Z}$-splitting complex}
\label{subsec:CyclicS}
Let $\mathcal{Z}$ be the collection of subgroups of $\free$ that are
either trivial or cyclic.  We denote by $\mathcal{Z}^{\text{max}}$ the
collection of elements of $\mathcal{Z}$ which are either trivial or
closed under taking roots.  We will use the notation
$\mathcal{Z}^{\m}$ to mean either $\mathcal{Z}$ or
$\mathcal{Z}^{\text{max}}$.  A \emph{$\mathcal{Z}^{\m}$-splitting} is
a minimal, simplicial $\free$-tree whose edge stabilizers belong to
the set $\mathcal{Z}^{\m}$; it is a \emph{one-edge splitting} if there
is one $\free$ orbit of edges.  A \emph{cyclic splitting}
(resp. maximally-cyclic splitting) is a one-edge
$\mathcal{Z}$-splitting (resp. $\mathcal{Z}^{\text{max}}$-splitting)
whose edge stabilizer is infinite cyclic.  Two
$\mathcal{Z}^{\m}$-splittings are \emph{equivalent} if the
corresponding Bass-Serre trees are $\free$-equivariantly homeomorphic.
We will often blur the distinction between a splitting and its
Bass-Serre tree.

If $S$ is a one-edge free splitting (resp.\
$\mathcal{Z}^{\m}$-splitting) and $v$ is a vertex in the Bass-Serre
tree, then $\Stab(v)$ will be called a \emph{vertex group} of $S$.
Vertex groups of free splittings are free factors.

Given two $\mathcal{Z}^{\m}$-splittings $\overline{T}$ and $T$, we say
that $\overline{T}$ is a \emph{refinement} of $T$ if there is a
collapse map from $\overline{T}$ to $T$.  Two
$\mathcal{Z}^{\m}$-splittings $T$ and $T'$ are \emph{compatible} if
they have a common refinement, i.e., if there exists a tree that
collapses onto both $T$ and $T'$.  A tree $T$ is
\emph{$\mathcal{Z}^{\m}$-incompatible} if the set of
$\mathcal{Z}^{\m}$-splittings compatible with $T$ is empty.  The
(maximally-) cyclic splitting complex $\CyclicS^{\m}$ is the
simplicial complex whose vertices are equivalence classes of one-edge
$\mathcal{Z}^{\m}$-splittings and whose $k$-simplicies are collections
of $k+1$ pairwise compatible one-edge $\mathcal{Z}^{\m}$-splittings.
In \cite{M:CyclicS}, Mann showed that $\CyclicS$ is
$\delta$-hyperbolic.  More recently, Horbez used the same argument
\cite{H:CyclicSBoundary} to prove that $\CyclicS^{\text{max}}$ is
$\delta$-hyperbolic.

The results of Shenitzer, Stallings, Swarup
\cite{Shenitzer,S:Foldings,S:DecompositionFreeGroups} imply that every
one-edge cyclic splitting of $\free$ is obtained from a one-edge free
splitting of $\free$ by `edge folding' process described as follows.
Let $T$ be a free splitting of $\free$, let $v$ be a vertex of $T$ and
let $G_v$ be its stabilizer.  Consider $w \in G_v$ and $\la w \ra$,
the cyclic group generated by $w$. Construct a new $\free$-tree $T'$
by first choosing an edge $e$ incident at $v$, then, for every
$\gamma \in \free$, identifying $\gamma e$ with its orbit under
$\la \gamma w \gamma^{-1} \ra \subseteq G_{\gamma v}$.  The tree $T'$
has an edge with stabilizer equal to $\la w \ra$.  We say $T'$ is
obtained from $T$ by an equivariant \emph{edge fold}, or to be more
specific, we sometimes say that $T'$ is obtained from $T$ by
performing the \emph{edge fold corresponding to $\la w \ra$}.

\subsection{$\mathcal{Z}$-averse trees and boundary of $\CyclicS$}
\label{sec:Z-averse} 
A tree $T$ in $\cvclo$ is called $\mathcal{Z}^{\m}$-averse
\cite[Definition 4.2]{H:CyclicSBoundary} if there is no finite chain
of compatibility between $T$ and a $\mathcal{Z}^{\m}$-splitting: that
is, if there is no finite set of trees ($T=T_0,T_1,\ldots,T_k=T'$) in
$\cvclo$ such that $T'$ is a $\mathcal{Z}^{\m}$-splitting and for each
$i\in \{0,\ldots,k-1\}$, the trees $T_i$ and $T_{i+1}$ are compatible.
Two $\mathcal{Z}^{\m}$-averse trees, $T,T'$, are called
\emph{equivalent} if there is a finite chain of compatible trees in
$\cvclo$ relating $T$ to $T'$ as above.  The reader will note that the
notions of $\mathcal{Z}^{\m}$-compatibility and
$\mathcal{Z}^{\m}$-aversity are independent of the homothety class of
$T$; in particular, it makes sense to say that a tree in $\CVclo$ is
$\mathcal{Z}$-averse, or that two trees in $\CVclo$ are equivalent.
We denote by $\mathcal{X}^{\m}$ (resp.\ $\mathbb{P}\mathcal{X}^{\m}$)
the subspace of $\cvclo$ (resp.\ $\CVclo$) consisting of
$\mathcal{Z}^{\m}$-averse trees.

There is a natural map from a subset of $\CVbound$ to the Gromov
boundary of $\CyclicS^{\m}$ relating the geometries at infinity of
these two spaces, which we now describe.  There is a map
$\psi^{\m}\colon\CV\to\CyclicS^{\m}$, which extends to the set of
simplicial trees in $\cvclo$ with trivial edge stabilizers, defined by
choosing a one-edge collapse of every simplicial tree in $\CV$.  This
map is not quite $\Out(\free)$-equivariant because we must make
choices, however differing choices change distances by at most 2.  The
following theorem due to Horbez describes the boundary of the free
splitting complex.

\begin{theorem}[{\cite[Theorem 0.1]{H:CyclicSBoundary}}]
  \label{thm:horbez}
  There is a unique $\Out(\free)$-equivariant homeomorphism
  \begin{equation*}
    \partial \psi^{\m} \colon \mathcal{X}^{\m} / \sim\quad \longrightarrow\quad \partial_{\infty}\CyclicS^{\m}
  \end{equation*}
  so that for all $T \in \mathcal{X}^{\m}$ and all sequences
  $(T_n) \in \cv^{\mathbb{N}}$ converging to $T$, the sequence
  $(\psi^{\m} (T_n))_{n \in \mathbb{N}}$ converges to $\psi(T)$.
\end{theorem}

Given a tree $T \in \cvclo$, a $\mathcal{Z}^{\m}$-splitting $S$ is called a \emph{reducing splitting} for $T$, if $S$ is compatible with some $T' \in \cvclo$, which is itself compatible with $T$.  
\subsection{Lines and Laminations}\label{sec:lam}
We briefly recall some definitions, but the reader is directed to
\cite{BFH:Tits} for details. The \emph{space of abstract lines},
$\widetilde{\mathcal{B}}=(\partial\free\times\partial\free-\Delta)
/\mathbb{Z}_2$ is the set of unordered distinct pairs of points in the
boundary of $\free$ and is equipped with the natural
(subspace/product/quotient) topology.  The quotient of
$\widetilde{\mathcal{B}}$ by the natural $\free$ action is \emph{the
  space of lines in $\mathcal{R}$} and is called $\mathcal{B}$.  It is
endowed with the quotient topology, which satisfies none of the
separation axioms.  Points in $\mathcal{B}$ and
$\widetilde{\mathcal{B}}$ will be called lines.

A closed subset $\Lambda$ of $\mathcal{B}$ is an \emph{attracting
  lamination for $\phi$} if it is the closure of a single line
$\beta$ that is \emph{bireccurrent} (every finite subpath $\sigma$ of
$\beta$ occurs infinitely many times as an unoriented subpath of each
end of $\beta$), has an \emph{attracting neigborhood} (there is some
open $U\ni\beta$ so that $\phi^k(\gamma)\to\beta$ for all
$\gamma\in U$), and is not carried by a rank one $\phi$-periodic
free factor.  The lines in $\Lambda$ satisfying the
above properties are called the \emph{generic leaves} of $\Lambda$.

A subgroup $A$ of $\free$ determines a subset of
the boundary of $\free$ called $\partial A\subset\partial\free$.  We
say that $A$ \emph{carries} a line $\beta$ if there is some lift
$\widetilde{\beta}$ whose endpoints are in $\partial A$.  We then say
that the $A$ \emph{carries} the lamination $\Lambda$ if $A$ carries
some (any) generic leaf of $\Lambda$.  A lamination $\Lambda$ is said
to be \emph{filling} (resp.\ \emph{$\mathcal{Z}^{\m}$-filling}) if
$\Lambda$ is not carried by any vertex group of any free (resp.\ $\mathcal{Z}^{\m}$-)
splitting.

Let $\pi_A\colon G_A\to \mathcal{R}$ be the immersion from the core of the
cover of $\mathcal{R}$ corresponding to the subgroup $A$ and let
$\beta$ be a line.  Then clearly $\beta$ is carried by $A$ if and only
if there exist immersions $\rho_A\colon \mathbb{R}\to G_A$ and
$\rho\colon \mathbb{R}\to \mathcal{R}$ such that $\rho=\pi_A\rho_A$.
If we further assume that $A$ is finitely generated, it's easy to see
that $\beta$ is carried by $A$ if and only if every finite subsegment
of $\beta$ can be immersed into $G_A$.

\section{Folding in the boundary of outer space}
\label{sec:folding-seq}
Throughout this section, $\phi$ will be an outer automorphism with a
$\mathcal{Z}^{\m}$-filling lamination $\lamination$.  Our first goal
is to extract from $\phi$ a folding path converging to a tree in
$\CVbound$ which ``witnesses'' the lamination $\lamination$. The
automorphism $\phi$ is fully irreducible relative to some maximal
$\phi$-invariant free factor system $\ffa$. Since $\phi$ has a filling
lamination, $\ffa$ is not an exceptional free factor system, that is,
is it not of the form $\{A\}$ or $\{A_1, A_2\}$ where
$\free = A \ast \mathbb{Z}$ or $\free = A_1 \ast A_2$. Let
$f\colon T \to T$ be the universal cover of a relative train track
representative of $\phi$ realizing the invariant free factor system
$\ffa$.  Let $G=T/\free$ be the quotient graph, which comes with a
filtration $\Rfltr$ such that $\mathcal{F}(G_{r-1}) = \ffa$ and $H_r$
is an EG stratum with Perron-Frobenius eigenvalue $\PFevalue$.  Let
$T_r$ (resp.\ $T_{r-1}$) denote the full preimage of $H_r$ (resp.\
$G_{r-1}$) under the quotient map $T\to G$.  We endow $G$ (and hence
$T$) with a metric by declaring all edges to have length $1$.  We will
henceforth consider $T$ as a point in unprojectivized outer space
$\cv$, whereby $f$ may be thought of as an $\free$-equivariant map
$T\to T\cdot \phi$.

Let $T_0'$ be the tree obtained from $T$ by equivariantly collapsing
the $\ffa$-minimal subtree.  Our present aim is to construct a folding
path ending at
$\StableTree := \lim_{n \to \infty} T_0' \phi^n/ \PFevalue^n$.  To
accomplish this, we will construct simplicial trees $T_0,T_1$ and
define an optimal morphism $f_0\colon T_0\to T_1$.  From this we will
obtain a periodic canonical optimal folding path $(f_t)_{t\in [0,L]}$
which will end at $\StableTree$.  It is worth noting that the natural
map $f_0'\colon T_0'\to T_0'\phi$ induced by $f$ is neither optimal
nor a morphism as there may be non-degenerate intervals which are
mapped to points.

We would like to remark that existence of an optimal morphism which is a train track map
representing a relative fully irreducible outer automorphism is a
special case of the results of \cite{FM:TTforRelativeOuterSpace} and
\cite{Meinert}, for free products and deformation spaces,
respectively.  The authors of \cite{FM:TTforRelativeOuterSpace}
develop metric theory for relative outer space for free products which
is then used to show the existence of optimal maps. This requires
considerable amount of work due to lack of applicability of
Arzela-Ascoli theorem in this setting. In what follows, we provide a
shorter proof of existence of a train track map representing $\phi$ in
the context of free groups.
 
\subsection*{Constructing $T_0$}
The following is based on the construction in the proof of \cite[Lemma
5.10]{BH:TrainTracks}.  Define a measure $\mu$ on $T$ with support
contained in the set $\{x \in T_r : f^k(x) \in T_r$ for all
$k\geq 0 \}$ as follows: choose a Perron-Frobenius eigenvector
$\vec{v}$ corresponding to the PF eigenvalue $\PFevalue$. For an edge
$e$ in $T_r$, let $\mu(e) = v_e$ where $v_e$ is the component of
$\vec{v}$ corresponding to $e$.  Define $\mu(e) = 0$ for all edges
$e \in T_{r-1}$. Let $V$ be the set of vertices of $T$ and let
$V_m :=\{x \in T : f^m(x) \in V\}$.  Subdividing $T$ at $V_m$ divides
each edge into segments that map to edge paths under $f^m$.  If $a$ is
such a segment then define $\mu(a) = \mu(f^m(a))/\PFevalue^m$.  The
definition of $\mu$ together with the fact that relative train track
maps take $r$-legal paths to $r$-legal paths implies:

\begin{lemma}\label{lem:mu-prop}
  If $[x,y]$ is an $r$-legal path in $T$, then
  $\mu(f_{\#}([x,y]))=\PFevalue \mu([x,y])$.  If $[x,y]$ contains an
  initial or terminal segment of some edge in $T_r$, then
  $\mu([x,y])>0$.
\end{lemma}

The measure $\mu$ defines a pseudometric $d_\mu$ on $T$.  Collapsing
the sets of $\mu$-measure zero to make $d_\mu$ into a metric, we
obtain a tree $T_0$.  Let $p\colon T\to T_0$ be the collapse map.

\begin{lemma}\label{lem:simplicial}
  $T_0$ is simplicial.
\end{lemma}
\begin{proof}
  We will show that the $\free$-orbit of any point in $T_0$ must be
  discrete.  Let $x\in T_0$ and choose a point
  $\tilde{x}\in p^{-1}(x)$.  The $\free$-orbit of $\tilde{x}$ in $T$
  is discrete, and to understand the orbit of $x$, we need only
  understand $\mu([\tilde{x},g\tilde{x}])$ for $g\in\free$.  If
  $[\tilde{x},g\tilde{x}]$ contains no edges in $T_r$, then
  $\mu([\tilde{x},g\tilde{x}])=0$, in which case $g\in\Stab(x)$.
  Otherwise, the segment contains an edge in $T_r$, and hence has
  positive $\mu$-measure.  Since there are only finitely many
  $\free$-orbits of edges in $T_r$, there is a lower bound on the
  $\mu$-measure of $[\tilde{x},g\tilde{x}]$.  Hence, there is a lower
  bound on $d_{T_0}(x,gx)$.  This concludes the proof.
\end{proof}

The trees $T_0$ and $T_0'$ are $\free$-equivariantly homeomorphic.
The problem with $T_0'$ is that the ``obvious'' map
$f_0'\colon T_0'\to T_0'\phi$ sends nondegenerate segments to points
and, because of that, is not useful for making a folding path.  The
map $f_0$ defined in the sequel is an improvement because it can be
used to construct a folding path.

\subsection*{Defining $f_0 \colon T_0 \to T_1$}
Let $T_1$ be the tree $\PFevalue^{-1}T_0\cdot\phi$: the leading
coefficient indicates that the metric has been scaled by
$\PFevalue^{-1}$.  The relative train track map
$f\colon T\to T\cdot\phi$ naturally induces a map
$f_0\colon T_0\to T_1$.  For each $x \in T_0$, its pre-image
$p^{-1}(x)$ is a connected subtree of $T$ with $\mu$-measure zero.
The definition of $\mu$ guarantees that the $f$-image of this set is
also connected and has $\mu$-measure zero.  Therefore
$p\circ f\circ p^{-1}(x)$ is a single point in $T_0\cdot\phi$, which
is identified with $T_1$ and we define $f_0:=p\circ f\circ p^{-1}$.

\begin{lemma}
  $f_0$ is an optimal morphism.
\end{lemma}

\begin{proof}
  We first show that $f_0$ is a morphism, which will follow from the
  definition of $\mu$ and properties of relative train track
  maps. Given a non-degenerate segment $[x,x']$ in $T_0$, choose
  $\tilde{x}\in p^{-1}(x)$ and $\tilde{x}'\in p^{-1}(x')$.  The
  intersection of $[\tilde{x},\tilde{x}']$ with the vertices of $T$ is
  a finite set $\{\tilde{x}_1,\ldots,\tilde{x}_{k-1}\}$.  Let
  $\tilde{x}_0:=\tilde{x}$ and $\tilde{x}_k:=\tilde{x}'$.  Taking the
  $p$-image of $\tilde{x}_i$ for $i\in\{0,\ldots,k\}$ yields a
  subdivision of $[x,x']$ into finitely many subsegments
  $[x_i,x_{i+1}]$, some of which may be degenerate.  We will ignore
  the degenerate subdivisions: they occur as the projections of edges
  in $T_{r-1}$ (all of which have $\mu$-measure zero).

  We claim that $f_0$ is an isometry in restriction to each of these
  subsegments.  Indeed, let $e=[\tilde{x}_i,\tilde{x}_{i+1}]$ be an
  edge in $T$.  Assume without loss of generality that
  $x_i\neq x_{i+1}$ so that $\mu(e)\neq 0$ and $e$ is therefore an
  edge in $T_r$.  It is an immediate consequence of Lemma
  \ref{lem:mu-prop} that for each $y\in e$,
  $\mu([f(\tilde{x}_i), f(y)])=\PFevalue \mu([\tilde{x}_i,y])$ and
  hence that $f_0$ is an isometry in restriction to $[x_i,x_{i+1}]$.

  We now address the optimality of $f_0$.  There are three types of
  points to consider: points in the interior of an edge, vertices with
  trivial stabilizer, and vertices with non-trivial stabilizer.  We
  have already established that $f_0$ is an isometry in restriction to
  edges, so there are two gates at each $x\in T_0$ contained in the
  interior of an edge.  If $x\in T_0$ is a vertex with trivial
  stabilizer, then $p^{-1}(x)$ is a vertex (Lemma \ref{lem:mu-prop})
  contained in $T_r \smallsetminus T_{r-1}$.  As $f$ is a relative
  train track map, there are at least two gates at $p^{-1}(x)$ and
  each is necessarily contained in $T_r$.  A short path in $T$
  containing $p^{-1}(x)$ entering through the first gate and leaving
  through the second will be legal.  Lemma \ref{lem:mu-prop} again
  gives that $f_0$ is an isometry in restriction to such a path, so
  there are at least two gates at $x$.

  Now let $x\in T_0$ be a vertex with non-trivial stabilizer.  Then
  $p^{-1}(x)$ is a subtree which is the inverse image of a component
  of $G_{r-1}$ under the quotient map $T\to G$.  Let
  $\tilde{x},\tilde{x}'\in p^{-1}(x)$ be distinct vertices in
  $T_r\cap T_{r-1}$ and let $d$ (resp.\ $d'$) be a direction based at
  $\tilde{x}$ (resp.\ $\tilde{x}'$) corresponding to an edge $e$
  (resp.\ $e'$) in $T_r$.  Lemma \ref{lem:mu-prop} provides that $d$
  and $d'$ determine distinct directions at $x$.  As mixed turns are
  legal, the path $\overline{e}\cup[\tilde{x},\tilde{x}']\cup e'$ in
  $T$ is $r$-legal.  A final application of Lemma \ref{lem:mu-prop}
  gives that the restriction of $f_0$ to the $p$-image of this path is
  an isometry, and hence that there are at least two directions at
  $x$.
\end{proof}

The reader will note that we have actually proved

\begin{lemma}
  $f_0$ is a train track map.
\end{lemma}

As $T_0$ and $T_0'$ are $\free$-equivariantly homeomorphic, there is a
bijection between ($\free$-orbits of) edges of each.  It is easily
verified that the transition matrix of $f_0$ and that of $f$ are
equal.  In particular, we will speak of edges, transition matrices, PF
eigenvalues, and related notions for $f_0\colon T_0\to T_1$, without
reference this bijection.

Next, we use $f_0$ to construct a folding path starting at
$S_0:=T_0$.  This folding path will converge in $\CVbound$ to a tree
$S_L$.  We then prove that $S_L$ is in fact the tree $\StableTree$
defined above.

\subsection*{Folding $T_0$}
Applying the canonical folding path construction, we obtain a folding
path $(S_t)_{t\in[0,L_1]}$ guided by $f_0\colon T_0\to T_1$ which
begins at $T_0=S_0$ and ends at $T_1=S_{L_1}$, where $L_1 = \frac{1}{2}\BBT(f_0)$.
Adapting a construction of Handel-Mosher \cite[Section 7.1]{HM:Axes},
we now extend this to a \emph{periodic fold path guided by $f_0$}.
For each $i\in\mathbb{N}$, let $T_{i}=\PFevalue^{-i}T_0\cdot\phi^i$,
whence we have optimal morphisms $f_i\colon T_i\to T_{i+1}$ satisfying
$\BBT(f_i)=\PFevalue^{-i}\BBT(f_0)$.  For each $i$, inductively define
$L_i:=L_{i-1}+\frac{1}{2}\BBT(f_{i-1})$ and extend the folding path (which
has so far been defined on $[0,L_{i-1}]$) using $f_{i-1}$ to a folding
path $(S_t)_{t\in[0,L_i]}$.  Define $L:=\lim_{i\to\infty} L_i$, which
is finite as $\BBT(f_i)$ is a geometric sequence.  We have thus
defined the trees $(S_t)_{t\in[0,L)}$.

The notation here is less than ideal.  In the above,
$(T_i)_{i\in \mathbb{N}}$ is used for the trees
$\PFevalue^{-i}T_0\cdot \phi^i$, while $(S_t)_{t\in [0,L)}$ denotes a
continuous folding path which is folded at constant speed.  The reason
for the differing names ($S$ and $T$) is simply that the
parameterizations differ; in particular $S_{L_i}=T_i$.

We now describe the maps $f_{t,s}$ for $s,t\in [0,L)$ with $s<t$.
Indeed, given $s,t$, there is a natural choice of a map
$f_{t,s}\colon S_s\to S_t$.  Suppose $s\in [L_i,L_{i+1})$ and
$t\in [L_j,L_{j+1})$.  Then
\begin{equation*}
  f_{t,s}:=f_{t,L_j}\circ f_{j-1}\circ f_{j-2}\circ\ldots\circ f_{i+1}\circ f_{L_{i+1},s}
\end{equation*}
The semi-flow property for the connection data follows from the
definitions.  Though our setting differs slightly from that of
\cite{BF:FreeFactorComplex}, Proposition 2.2 (5) can still be applied
to give that each tree $S_t$ has a well defined train track structure.

Along with the connection data, the fold path $(S_t)_{t\in[0,L)}$
forms a directed system in the category of $\free$-equivariant metric
spaces and distance non-increasing maps.  As direct limits exist in
this category, let $S_L:=\varinjlim S_t$ and let $f_{L,t}$ be the
direct limit maps.  The proof of the following proposition is
contained in Section 7.3 of \cite{HM:Axes}, though it is not stated in
this way.  While Handel-Mosher deal with trees in $\cv$ rather than
$\partial\CV$, the reader will easily verify that their proof goes
through directly in our setting.

\begin{proposition}[{\cite{HM:Axes}}]
  $S_L$ is a non-trivial, minimal, $\mathbb{R}$-tree.  Moreover $S_t$
  converges to $S_L$ in the length function topology.
\end{proposition}

We have now described two trees in the boundary of outer space:
$\StableTree=\lim_{n\to\infty} T_0'\phi^n$ and $S_L$.  We observe that
both $S_0$ and $T_0'$ are points in the relative outer space $\relcv$,
which inherits the subspace topology from $\CVclo$.  Moreover, $\phi$
is fully irreducible relative to $\mathcal{A}$, and as such, it acts
with north-south dynamics on $\relCVClo$ \cite{G:Loxodromic}. Recall
that for each $i\in\mathbb{N}$,
$S_{L_i}=\PFevalue^{-i}S_0\cdot\phi^i$, and that $L_i\to L$.  As $S_L$
is the limit of the fold path $(S_t)_{t\in [0,L)}$, we conclude

\begin{lemma}
  $S_L=\StableTree$.
\end{lemma}

We conclude this section with a lemma.

\begin{lemma}\label{lem:simplicial-path}
  For all $t\in [0,L)$, the tree $S_t$ is simplicial.
\end{lemma}

\begin{proof}
  Let $t\in[0,L)$.  If $t=0$, Lemma~\ref{lem:simplicial} provides that
  $S_0$ is simplicial.  Since $S_{L_i}=\PFevalue^{-i}S_0\cdot \phi^i$,
  the lemma holds when $t=L_i$ for some $i\in\mathbb{N}$.  The other
  possibility is that $t\in (L_i,L_{i+1})$ for some $i$.  Since both
  $S_{L_i}$ and $S_{L_{i+1}}$ have trivial edge stabilizers,
  Proposition 1.1 of \cite{H:CyclicSBoundary} applies to the folding
  path guided by $f_i$ and allows one to concluded that all trees
  $S_t$, $t\in [L_i,L_{i+1}]$ are simplicial, as desired.
\end{proof}
\section{Stable tree is $\mathcal{Z}^{\m}$-averse}
\label{sec:mixing}
Our present aim is to understand $\StableTree$; we would like to show
that it is $\mathcal{Z}^{\m}$-averse.  In this section, we will use
the leaves of the topmost lamination $\Lambda_\phi^+$ to construct a
transverse covering of $\StableTree$, then use the transverse covering
to achieve our goal.

\begin{definition}
  Let $G$ be a group and $T$ be an $\mathbb{R}$-tree equipped with an
  action of $G$ by isometries; and let $K\subseteq T$ be a subtree.
  We say that the action $G \curvearrowright T$ is \emph{supported on
    $K$} if for any finite arc $J\subseteq T$, there are
  $g_1,\ldots, g_r\in G$ such that
  $I\subseteq g_1 K\cup\ldots\cup g_r K$.
\end{definition}

Let $I_0$ be a segment of a leaf of the lamination $\lamination$ in
$S_0$.  Define the arc $I_t$ in $S_t$ by $I_t:=f_{t,0}(I_0)$.  We will
denote $I_L$ simply by $I$ and we will call any segment in
$\StableTree$ obtained in this way \emph{a segment of a leaf of
  $\lamination$}.

\begin{lemma}\label{lem:support}
  The action $\free\curvearrowright \StableTree$ is supported on
  $I$.
\end{lemma}
\begin{proof}
  Let $I=[x,y]$ and let $J=[x',y']$ be a nondegenerate arc in
  $\StableTree$.  The construction in Section \ref{sec:folding-seq}
  provides an optimal folding path $(S_t)_{t \in [0,L]}$, and optimal
  morphisms $f_{s,t} \colon S_t \to S_s$ for all $s,t\in[0,L]$ with
  $s>t$ which satisfy the semi-flow property.  It follows easily from
  the definitions that for a folding path $(S_t)$ and any $z$ in
  $S_L = \StableTree$, the set $f^{-1}_{L,0}(z)$ is a discrete set of
  points in $S_0$.  Let $x_0'\in f^{-1}_{L,0}(x')$ and
  $y_0'\in f^{-1}_{L,0}(y')$ be points in $S_0$ chosen so that
  $(x_0',y_0')$ contains no points in
  $f^{-1}_{L,0}(x')\cup f^{-1}_{L,0}(y')$ and define
  $J_0=[x_0',y_0']$.  Since $I_0$ is legal, it is never folded under
  the maps $f_{t,0}$, so the corresponding property already holds for
  $I_0$.  Define the arc $J_t$ in $T_t$ by $J_t:=[f_{t,0}(J_0)]$.  The
  definitions of $I_0$ and $J_0$ guarantee that $[f_{L,0}(I_0)]=I$ and
  similarly for $J_0$.  The semiflow property of the maps $f_{s,t}$
  gives that for all $s,t \in [0,L]$ with $s>t$, we have
  $[f_{s,t}(I_t)] = I_s$ (resp.\ $[f_{s,t}(J_t)] = J_s$).

  Since $I_0$ is a leaf segment and therefore legal with respect to
  the train track structure on $S_0$, it is never folded under the
  maps $f_{t,0}$.  In particular, the length of $I_t$ is constant in
  $t$.  The maximum length of any edge in $S_t$ tends to $0$ as
  $t\to L$ because edge lengths can only decrease along the fold path
  and the metric in $S_{L_i}$ has been scaled by $\PFevalue^{-i}$.
  Thus, for sufficiently large $t$, $I_t$ crosses an entire edge of
  $S_t$.  Irreducibility of the transition matrix for $f_0$ implies
  that by further enlarging $t$, we may assume that $I_t$ crosses
  an edge from every $\free$-orbit of edges in $S_t$.

  We are now ready to complete the proof.  Indeed, write $J_t$ as an
  edge path $J_t=e_0e_1\ldots e_k$ in $S_t$ (the first and last edges
  may be partial edges).  Since $I_t$ crosses every $\free$-orbit of
  edges in $S_t$, there exist $g_0,\ldots ,g_k\in \free$ so that for
  all $j$, $g_jI_t$ crosses the edge $e_j$.  Now we simply use
  $\free$-equivariance of the maps $f_{L,t}$ to conclude that
  \begin{equation*}
    f_{L,t}(J_t)\subseteq g_0f_{L,t}(I_t)\cup
    g_1f_{L,t}(I_t)\cup\ldots\cup g_kf_{L,t}(I_t)
  \end{equation*}
  As $I_t$ is legal, $f_{L,t}(I_t)=I$.  While $J_t$ is not necessarily
  legal, it's still true that
  $J=[f_{L,t}(J_t)] \subseteq f_{L,t}(J_t)$, completing the proof.
\end{proof}

\subsection{Mixing and indecomposable trees}
A tree $T \in \CVclo$ is \emph{mixing} if for all finite subarcs
$I, J \subset T$, there exist $g_0, \ldots, g_k \in \free$ such that
$J \subseteq g_0I \cup g_1I \cup \cdots \cup g_kI$ and
$g_jI \cap g_{j+1}I \neq \emptyset$ for all
$j \in \{0, \ldots ,k-1\}$.  A tree $T \in \CVclo$ is called
\emph{indecomposable} \cite{G:Actions} if it is mixing and the $g_j$'s
can be chosen so that $g_j I\cap g_{j+1}I$ is a non-degenerate arc for
each $j\in\{0,\ldots,k-1\}$.

\begin{lemma}\label{lem:mixing}
  $\StableTree$ is mixing.
\end{lemma}

\begin{proof}
  The proof is similar to that of Lemma \ref{lem:support}, so we will
  retain our notation from that proof.  Indeed, it's clearly enough to
  show that every arc $J$ can be covered by finitely many translates
  with non-empty overlap of the fixed arc $I$ and conversely that $I$
  can be covered similarly by translates of $J$.  Recall the cover of
  $J$ by translates of $I$ constructed in proof of Lemma
  \ref{lem:support}.  Since consecutive edges in the edge path of
  $J_t=e_0\ldots e_k$ intersect in a point, it follows that
  $g_jI_t\cap g_{j+1}I_t\neq \emptyset$ for all
  $j\in\{0,\ldots,k-1\}$.  Again, this behavior persists in the limit.

  Conversely, to see that $I$ can be covered by translates of $J$ we
  use essentially the same argument as before, only now there is a
  slight difficulty in producing an edge in some $J_t$ that
  isometrically embeds in the limit.  Now $J_t$ may have illegal
  turns, so we write $J_t$ as a concatenation of maximal legal
  subpaths, $J_t=J_t^0J_t^1\ldots J_t^k$.  Now $f_{L,t}(J_t)$ is a
  concatenation of the $f_{L,t}$-images of $J_t^i$, which are
  themselves segments in $S_L$.  Thus, the tightened image
  $J=[f_{L,t}(J_t)]$ is contained in the union
  $f_{L,t}(J_t^0)\cup\ldots\cup f_{L,t}(J_t^k)$.  Now choose an
  $i\in \{0,\ldots,k\}$ so that $J\cap f_{L,t}(J_t^i)$ is a
  non-degenerate subsegment of $J$ and replace $J$ by the subsegment
  $J'=J\cap f_{L,t}(J_t^i)$.  The proof of Lemma \ref{lem:support} can
  now be applied to $J'$, allowing us to conclude that $I$ can be
  covered by finitely many translates $J'$ with nonempty overlaps.  As
  $J'$ is a subsegment of $J$, the same finite set of group elements
  witnesses the fact that $I$ can be covered by finitely many
  translates $J$ with nonempty overlaps.
\end{proof}

\subsection{Transverse families and transverse coverings}
A subtree $Y$ of a tree $T$ is called \emph{closed} \cite[Definition
2.4]{G:LimitGroups} if $Y\cap\sigma$ is either empty or a path in $T$
for all paths $\sigma\subset T$; recall that paths are defined on
closed intervals.  A \emph{transverse family} \cite[Definition
4.6]{G:LimitGroups} of an $\mathbb{R}$-tree $T$ is a family
$\mathcal{Y}$ of non-degenerate closed subtrees of $T$ such that any
two distinct subtrees in $\mathcal{Y}$ intersect in at most one point.
If every path in $T$ is covered by finitely many subtrees in
$\mathcal{Y}$, then the transverse family is called a \emph{transverse
  covering}.

The idea of the following definition is to start with an interval and
``fill it out'' into an entire subtree by translating it around,
always requiring that overlaps are non-degenerate.

\begin{definition}[The transverse family generated by $\lamination$]
  Let $I=[x,y]$ be a segment of a leaf of $\lamination$ in
  $\StableTree$.  Define $Y_{I}$ as the union of all arcs $J$ such
  that there exists $g_0, \ldots, g_k \in \free$ satisfying:
  \begin{itemize}
  \item $J \subseteq g_0I \cup \cdots \cup g_k I$,
  \item $g_jI \cap g_{j+1}I$ is a non-degenerate segment for each
    $i\in\{0,\ldots,k-1\}$, and
  \item $g_0I\cap I$ is a non-degenerate segment.
  \end{itemize}
  It's immediate that the collection
  $\mathcal{Y} = \{gY_I\}_{g \in \free}$ is a transverse family in
  $\StableTree$ since, by definition, distinct $\free$-translates of
  $Y_I$ intersect in a point or not at all.  This construction is
  essentially due to Guirardel-Levitt.
\end{definition}

\begin{lemma}\label{lem:transverse-cover} With notation as above,
  $Y_I$ is indecomposable with respect to the $\Stab(Y_I)$ action.
  Moreover, $\mathcal{Y}=\{gY_I\}_{g\in\free}$ is a transverse
  covering of $\StableTree$.
\end{lemma}

\begin{proof}
  We first show that $Y_I$ is indecomposable.  Again the proof is
  similar to that of Lemmas \ref{lem:support} and \ref{lem:mixing}, so
  we will retain our notation from those proofs.  As before, it is
  enough to show that every arc $J\subseteq Y_I$ can be covered by
  finitely many translates with non-degenerate overlap of the fixed
  arc $I$, and conversely that $I$ can be covered by finitely many
  translates of $J$ with non-degenerate overlap.  The definition of
  $Y_I$ guarantees that $J$ can be covered by finitely many translates
  of $I$, so we are left to show the converse.

  First, replace $J$ by an appropriately chosen subinterval exactly as
  in the proof of Lemma \ref{lem:mixing}.  Now we run the proof of
  Lemma \ref{lem:support} with a minor modification.  Indeed, for
  $t\in [0,L)$, let $J_t$ and $I_t$ be as in that proof.  This time,
  choose $t$ large enough so that $I_t$ crosses every $\free$-orbit of
  turns taken by a leaf of $\lamination$.  By further enlarging $t$ if
  necessary, we may arrange that $J_t$ also crosses every turn taken
  by a leaf.  Write $I_t$ as an edge path $I_t=e_0e_1\ldots e_k$ in
  $S_t$, where the first and last edges may be partial edges.  Since
  $J_t$ crosses every $\free$-orbit of turns taken by a leaf in $S_t$,
  there exist $g_0,\ldots ,g_k\in \free$ so that for all
  $j\in \{0,\ldots,k-1\}$, $g_jJ_t$ crosses the edge path
  $e_je_{j+1}$.  Now we conclude exactly as before, using
  $\free$-equivariance of the maps $f_{L,t}$ to see that
  \begin{equation*}
    f_{L,t}(I_t)\subseteq g_0f_{L,t}(J_t)\cup
    g_1f_{L,t}(J_t)\cup\ldots\cup g_kf_{L,t}(J_t)
  \end{equation*}
  Since both $I_t$ and $J_t$ are legal, this set containment (and
  non-degeneracy of the overlaps) is unaffected by tightening and the
  proof is complete.

  To see that $\mathcal{Y}$ is a transverse covering we again
  reference the proof of Lemma \ref{lem:support}, which actually shows
  that every path in $\StableTree$ can be covered by finitely many
  trees in $\mathcal{Y}$.
\end{proof}

\begin{lemma}\label{lem:segments}
  Let $\beta$ be a generic leaf of $\lamination$ and let $J$ be a
  finite subsegment of a realization of $\beta$ in $\StableTree$.
  Then there exists $g\in\free$ which is contained in a conjugate of
  $\Stab(Y_I)$ and whose axis, $A_g$, in $\StableTree$ contains the
  segment $J$.
\end{lemma}
\begin{proof}
  We retain our notation from above, so that $J_t$ is a segment in
  $S_t$ which maps to $J$ under $f_{L,t}$.  We will denote the
  realization of $\beta$ in $S_t$ by $\beta_t$.  Choose $t$ large
  enough so that $J_t$ crosses every turn taken by $\beta_t$, then
  lengthen $J_t$ by following the leaf to arrange that both endpoints
  of $J_t$ are vertices in the same $\free$-orbit.  Write $J_t$ as an
  edge path $J_t=e_0e_1\ldots e_k$.  If necessary, further lengthen
  $J_t$ (again following $\beta_t$) to arrange that the turn
  $\{e_0,\overline{e_k}\}$ is taken by a leaf.  Let $x_t$ (resp.\
  $y_t$) be the initial (resp.\ terminal) endpoint of $J_t$.

  Now let $g\in\free$ be a group element taking $x_t$ to $y_t$.  After
  postcomposing with an element of $\Stab(y_t)$ if necessary, we may
  assume that the turn $\{\overline{e_k}, g(e_0)\}$ is taken by a
  generic leaf of $\lamination$.  We claim that the axis of $g$ in
  $S_t$ crosses $J_t$.  Indeed, to get from $x_t$ to $y_t$, one
  traverses the edge path $e_0e_1\ldots e_k$.  Thus, to get from
  $y_t=g\cdot x_t$ to $g\cdot y_t=g^2\cdot x_t$, one traverses the
  same (up to $\free$-orbit) edge path.  As $e_0\neq \overline{e_k}$
  and $S_t$ is a tree, we have that
  $d(x_t,g^2\cdot x_t)=2 d(x_t,g\cdot x_t)$.  It is an elementary
  exercise to show that this is equivalent to $x$ being on the axis of
  $g$.  Both $\beta_t$ and the axis of $g$ are legal, so the
  restriction of $f_{L,t}$ to each is an immersion.  Thus, we can push
  this picture forward to the limit using $f_{L,t}$ to reach the
  desired conclusion.

  We've seen that any realization of $\beta$ in $\StableTree$ is
  contained in a single $\free$-translate of $Y_I$.  As we have
  arranged that every turn taken by the axis of $g$ in $S_t$ is also
  taken by a leaf, the argument given in the proof of Lemma
  \ref{lem:transverse-cover} allows us to conclude that $A_g$ is
  contained in a single $\free$-translate of $Y_I$.  Thus $g$ is
  contained in a conjugate of $\Stab(Y_I)$, as desired.
\end{proof}

For convenience of the reader, we recall two essential facts:
\begin{proposition}[{\cite[Proposition 4.3,
    4.27]{H:CyclicSBoundary}}]\label{prop:mix-z}
  If $T\in\cvclo$ is mixing, then $T$ is $\mathcal{Z}^{\m}$-averse if
  and only if $T$ is $\mathcal{Z}^{\m}$-incompatible.
\end{proposition}

\begin{lemma}[{\cite[Lemma 1.18]{G:Actions}}]\label{lem:minimal-subtree}
  Let $T\in\cvclo$ be compatible with a $\mathcal{Z}^{\m}$-splitting,
  $S$.  Let $H\subseteq \free$ be a subgroup, such that the
  $H$-minimal subtree $T_H$ of $T$ is indecomposable.  Then $H$ is
  elliptic in $S$.
\end{lemma}

\begin{proposition}\label{prop:T-Z-averse}
  $\StableTree$ is $\mathcal{Z}^{\m}$-averse.
\end{proposition}
\begin{proof}
  We assume that $\StableTree$ is not $\mathcal{Z}^{\m}$-averse and
  argue towards a contradiction.  Indeed, as $\StableTree$ is mixing,
  Proposition \ref{prop:mix-z} implies that it is compatible with a
  $\mathcal{Z}^{\m}$-splitting $S$.  Now let $H=\Stab(Y_I)$.  If
  $Y_I=\StableTree$, then $H=\free$ and Lemma
  \ref{lem:minimal-subtree} gives that $\free$ is elliptic in $S$, a
  contradiction as $S$ is a nontrivial minimal splitting.

  The other possibility is that $Y_I$ is a proper subtree in
  $\StableTree$, and in this situation we argue that $\lamination$ is
  carried by a vertex group of $S$.  As above, we apply Lemma
  \ref{lem:minimal-subtree} to conclude that $H=\Stab(Y_I)$ is carried
  by a vertex group $A$ of the splitting $S$.  We have a tower of
  covers corresponding to subgroups as follows (we temporarily blur
  the distinction between $\free$ and the universal cover of
  $\mathcal{R}$)
  \begin{equation*}
    \free\overset{\pi_{H,\free}}{\longrightarrow}
    X_H\overset{\pi_{A,H}}{\longrightarrow}
    X_A\overset{\pi_{\mathcal{R},A}}{\longrightarrow}
    \mathcal{R}
  \end{equation*}
  We denote by $G_A$ and $G_H$ the core of the corresponding covers.

  Let $\beta$ be a generic leaf of $\lamination$.  Even though $H$ may
  not be finitely generated, we claim it is enough to show that every
  finite subsegment of $\beta$ can be immersed into $G_H$.  Indeed, by
  postcomposing these immersions with $\pi_{A,H}$ (also an immersion),
  we see that every finite subpath of $\beta$ can then be immersed
  into $G_A$.  Since $A$ is finitely generated, we conclude that
  $\beta$ can be immersed into $G_A$, and therefore that $\lamination$
  is carried by a vertex group of the cyclic splitting $S$.
  
  Let $h\colon \free\to \StableTree$ be an $\free$-equivariant map
  which is linear on edges and Lipschitz (it's easy to see that such
  maps exist).  Lemma 3.1 of \cite{BFH:Laminations} gives that
  $\BBT(h)$ is finite.  Color the line $\beta_L$ in $\StableTree$ red
  and let $\beta_\free$ be the realization of $\beta$ in $\free$.
  Pull back the coloring via $h$ to $\beta_\free$ as follows (keeping
  in mind the bounded cancellation): if $x\in\beta_\free$ is such that
  $h(x)$ is red, then color $x$ red, otherwise do not color $x$.  It's
  clear that both ends of $\beta_\free$ have red segments.

  Let $J_\free$ be a subsegment of $\beta_\free$.  Extend $J_\free$
  along $\beta_\free$ if necessary to ensure that both endpoints of
  $J_\free$ are red.  Define $J=h_\#(J_\free)$.  The fact that the
  endpoints of $J_\free$ are red ensures that $J$ is a subsegment of
  $\beta_L$.  Apply Lemma \ref{lem:segments} to obtain an element
  $g\in H$ whose axis contains $J$.  Color the axis of $g$ in
  $\StableTree$ blue.  Pull back this coloring to the axis of $g$ in
  $\free$ exactly as above.  Equivariance of $h$, coupled with the
  fact that $g$ is not elliptic in $\free$ or $\StableTree$, implies
  that every subray of the axis of $g$ in $\free$ contains blue
  points.  In particular, there are blue points on either side of
  $J_\free$.  Thus the axis of $g$ in $\free$ contains the prescribed
  segment $J_\free$.  It's now evident that $J_\free$ is contained in
  the $H$-minimal subtree of $\free$.  This implies that
  $\pi_{H,\free}(J_\free)$ is contained in the core $G_H$ of the
  cover, completing the proof.
\end{proof}

\section{Filling but not $\mathcal{Z}^{\m}$-filling laminations}
\label{sec:filling not Z-filling}
In this section, we study filling laminations which are not
$\mathcal{Z}^{\m}$-filling.  We then use this understanding to
establish the following proposition, which is a restatement of the
second claim in Theorem~\ref{thm:loxodromics}.  This section concludes
with a proof of the first statement in Theorem~\ref{thm:loxodromics}.

\begin{proposition}\label{prop:v-group}
  Let $\phi$ be an automorphism with a filling lamination
  $\lamination$ that is not $\mathcal{Z}^{\m}$-filling, so that
  $\lamination$ is carried by a vertex group of a (maximally-) cyclic
  splitting $S$. Then there is a (maximally-) cyclic splitting $S'$
  that is fixed by a power of $\phi$.
\end{proposition}

The splitting $S'$ is canonical in the sense that the vertex group
which carries $\lamination$ is as small as possible.  The proof of
Proposition \ref{prop:v-group} will require an excursion into the
theory of JSJ-decompositions; the reader is referred to \cite{FP:JSJ}
for details about JSJ theory.

We say a lamination is \emph{elliptic} in an $\free$-tree $T$ if it is
is carried by a vertex stabilizer of $T$. Let $\mathfrak{S}$ be the
set of all one-edge $\mathcal{Z}^{\m}$-splittings in which the
lamination $\lamination$ is elliptic. Since $\lamination$ is filling,
the set $\mathfrak{S}$ does not contain any free splittings.

\begin{definition}[Types of pairs of splittings \cite{RipsSela}]
  Let $S = A \ast_{C} B$ (or $A\ast_{C}$) and $S'=A'\ast_{C'}B'$ (or
  $A'\ast_{C'}$) be one-edge cyclic splittings with corresponding
  Bass-Serre trees $T$ and $T'$. We say $S$ is \emph{hyperbolic} with
  respect to $S'$ if there is an element $c \in C$ that acts
  hyperbolically on $T'$. We say $S$ is \emph{elliptic} with respect
  to $S'$ if $C$ is fixes a point of $T'$. We say this pair is
  \emph{hyperbolic-hyperbolic} if each splitting is hyperbolic with
  respect to the other. We define elliptic-elliptic,
  hyperbolic-elliptic and elliptic-hyperbolic splittings similarly.
\end{definition}

\begin{lemma}\label{lem:Lamination-Elliptic}
  With notation as above, suppose that $S,S'\in\mathfrak{S}$, and
  assume without loss that $\lamination$ is carried by the vertex
  groups $A$ and $A'$. Then $\lamination$ is elliptic in the minimal
  subtree of $A$ in $T'$, denoted $T'_{A}$ and in the minimal subtree
  of $A'$ in $T$, denoted $T_{A'}$.
\end{lemma}
\begin{proof}
  Since $A$ and $A'$ both carry $\lamination$, their intersection
  $A \cap A'$ also carries
  $\lamination$. 
  The vertex stabilizers of $T_{A'}$ are precisely the intersection of
  vertex stabilizers of $T$ with $A'$, namely the conjugates of
  $A\cap A'$.  Thus $\lamination$ is carried by a vertex group of
  $T_{A'}$.
\end{proof}

\begin{lemma}\label{lem:HH-or-EE}
  With notation as above, suppose that $S,S'$ are one-edge
  $\mathcal{Z}^{\m}$-splittings in $\mathfrak{S}$. Then $S$ and $S'$
  are either hyperbolic-hyperbolic or elliptic-elliptic.  \end{lemma}
\begin{proof}
  The following is based on the proof of \cite[Proposition
  2.2]{FP:JSJ}. We will address the case that both the splittings are
  free products with amalgamations; when one or both are HNN
  extensions, the proof is similar. Toward a contradiction, suppose
  some element of $C$ acts hyperbolically in $T'$ and that $C'$ is
  elliptic in $T$.  Without loss of generality, we may assume that
  $C'$ fixes the vertex stabilized by $A$ in $T$. Suppose first that
  both $A'$ and $B'$ fix vertices in $T$. The two subgroups cannot fix
  the same vertex because they generate $\free$. On the other hand, if
  the vertices are distinct, then $C'$ fixes an edge in $T$.  Hence
  $C'$ must be a finite index subgroup of $C$, in contradiction to the
  assumption that $C$ is hyperbolic in $T'$.  Thus, one of $A'$ or
  $B'$ does not fix a vertex in $T$.

  Assume without loss that $A'$ does not fix a vertex of $T$.  The
  minimal subtree of $A'$ in $T$, denoted $T_{A'}$, gives a minimal
  splitting of $A'$ over an infinite index subgroup of $C$ (i.e. a
  free splitting).  Indeed, were $A'$ to split over a finite index
  subgroup $C_1$ of $C$, then $C_1$ would be elliptic in $T'$
  contradicting our assumption that $C$ is hyperbolic in $T'$.  As
  $C'$ is elliptic in $T$, it is also elliptic in $T_{A'}$.  Now
  blowup the vertex stabilized by $A'$ in $T'$ to the free splitting
  of $A'$ just obtained, and then collapse the edge stabilized by $C'$
  to get a free splitting $T_0$ of $\free$.  Then $B'$ is still
  elliptic in $T_0$.  If $\lamination$ is carried by $B'$, then
  $\lamination$ is elliptic in the free splitting $T_0$, a
  contradiction.  If $\lamination$ is carried by $A'$, then
  Lemma~\ref{lem:Lamination-Elliptic} implies that $\lamination$ is
  elliptic in $T_{A'}$.  Thus $\lamination$ is also elliptic in the
  free splitting $T_0$, again a contradiction.
\end{proof}

In \cite{FP:JSJ}, the existence of JSJ decompositions for splittings
with slender edge groups (\cite[Theorem 5.13]{FP:JSJ}) is established
via an iterative process: one starts with a pair of splittings, and
produces a new splitting which is a common refinement (in the case of
an elliptic-elliptic pair) \cite[Proposition 5.10]{FP:JSJ}, or an
enclosing subgroup \cite[Definition 4.5]{FP:JSJ} (in the case of a
hyperbolic-hyperbolic pair) \cite[Proposition 5.8]{FP:JSJ}.  One then
repeats this process for all the splittings under consideration, and
uses an accessibility result due to Bestvina-Feighn
\cite{BF:Accessibility} to conclude that the process stops after
finitely many iterations.  In order to use Fujiwara-Papasoglu's
techniques, we need only ensure that if two splittings belong to the
set $\mathfrak{S}$, then the splittings created in this process also
belong to $\mathfrak{S}$.  By examining the construction of an
enclosing subgroup for a pair of hyperbolic-hyperbolic splittings
\cite[Proposition 4.7]{FP:JSJ} and using
Lemma~\ref{lem:Lamination-Elliptic}, we see that the enclosing graph
decomposition of $\free$ for this pair of splittings indeed belongs to
$\mathfrak{S}$.  Similarly, Lemma~\ref{lem:Lamination-Elliptic}
implies that the refinement of two elliptic-elliptic splittings that
contained in $\mathfrak{S}$ is and is itself contained in
$\mathfrak{S}$.  This discussion implies that JSJ decompositions
exists for cyclic splittings of $\free$ in which $\lamination$ is
elliptic.

We conclude our foray into JSJ decompositions by using the theory of
deformation spaces \cite{Forester:Deformation,GL:DeformationSpaces} to
show that the set of JSJ splittings of $\free$ in which $\lamination$
is elliptic is finite.  By passing to a power, we will then obtain a
$\phi$-invariant splitting in $\mathfrak{S}$.

\begin{definition}[Slide moves {\cite[Section 7]{GL:DeformationSpaces}}] 
  Let $e=vw$ and $f = vu$ be adjacent edges in an $\free$-tree $T$
  such that the edge stabilizer of $f$, denoted $G_f$, is contained
  in $G_e$. Assume that $e$ and $f$ are not in the same orbit as
  non-oriented edges. Define a new tree $T'$ with the same vertex set
  as $T$ and replacing $f$ by an edge $f' = wu$ equivariantly. Then we
  say $f$ \emph{slides} across $e$. Often, a slide move is described
  on the quotient of $T$ by $\free$.
\end{definition}

\begin{definition}[{\cite{GL:DeformationSpaces,Forester:Deformation}}]
  The \emph{deformation space} $\mathcal{D}$ containing a tree $T$ is
  the set of all trees $T'$ such that there are equivariant maps from
  $T$ to $T'$ and from $T'$ to $T$, up to equivariant
  isometry. 
\end{definition}

\begin{definition}[{\cite{Forester:Deformation}}]
  A tree $T$ is reduced if no inclusion of an edge group into either
  of its vertex group is an isomorphism.
\end{definition}

\begin{theorem}[{\cite[Theorem 7.2]{GL:DeformationSpaces}}]
  \label{thm:slide moves}
  If $\mathcal{D}$ is a non-ascending deformation space, then any two
  reduced simplicial trees $T, T' \in \mathcal{D}$ may be connected by
  a finite sequence of slides.
\end{theorem}

Deformation spaces consisting of trees such that no edge stabilizer
properly contains a conjugate of itself are examples of non-ascending
deformation spaces \cite[Section 7]{GL:DeformationSpaces}.  We will
only be interested in such deformation spaces here.

\begin{lemma}\label{lem:slide moves}
  Given a reduced cyclic splitting $S$, there are only finitely many
  slide moves that can be performed on $S$.  Moreover, any sequence of
  slide moves starting at $S$ has bounded length. \end{lemma}
\begin{proof}
  The first statement follows from the fact that $S$ has finitely many
  orbits of edges.  For the second statement, first suppose that the
  splitting $S/\free$ does not have any loops or circuits.  Then it is
  clear that only finitely many slide moves can be performed on
  $S$. If $S$ has a loop, then we can slide an edge $f$ along the loop
  $e$ only once.  Indeed, we have $G_f \subseteq G_e$ and after sliding
  we have $G_{f'} \subseteq tG_et^{-1}$, where $t$ is the stable
  letter corresponding to the loop. Since $G_e \cong \mathbb{Z}$ and
  $G_e \cap tG_et^{-1} = 1$, $G_{f'} \not\subseteq G_e$ which prevents
  sliding of $f'$ over $e$.  The proof in the case of a circuit is
  similar.
\end{proof}

\begin{proof}[Proof of Proposition~\ref{prop:v-group}]
  By assumption, there exists a one-edge cyclic splitting $S$ such
  that $\lamination$ is elliptic in $S$.  The existence of JSJ
  decomposition for splittings in $\mathfrak{S}$ implies that the
  deformation space $\mathcal{D}$ for cyclic splittings in
  $\mathfrak{S}$ is non-empty.  Since the edge stabilizer of the trees
  in $\mathcal{D}$ is $\mathbb{Z}$, the space $\mathcal{D}$ is
  non-ascending.  Theorem~\ref{thm:slide moves} and
  Lemma~\ref{lem:slide moves} together imply that the set of reduced
  trees in $\mathcal{D}$ is finite.  As the set of reduced trees in
  $\mathcal{D}$ is $\phi$-invariant, passing to a power yields a
  reduced cyclic splitting $S'$ in $\mathcal{D}$ which is fixed by
  $\phi^k$.  The same argument works if $S$ is a maximally-cyclic
  splitting.
\end{proof}

\begin{proof}[Proof of Theorem~\ref{thm:loxodromics} (Loxodromic)] 
  Suppose that $\phi$ has a $\mathcal{Z}^{\m}$-filling lamination,
  whereby $\phi^{-1}$ does as well.  Applying Proposition
  \ref{prop:T-Z-averse} we conclude that both $\StableTree$ and
  $\UnstableTree$ are $\mathcal{Z}^{\m}$-averse.  We now argue that
  these trees determine distinct points in $\mathcal{X}^{\m}$.  We
  denote the dual lamination of a tree $T$ by $L(T)$
  \cite{CHL:DualLaminationsII}.  Since the attracting lamination
  $\lamination$ and the repelling lamination $\Rlamination$ are
  different, and $\MPlaminations \subseteq L(\PMTrees)$ and
  $\PMlaminations \nsubseteq L(\PMTrees)$, we have that $\StableTree$
  and $\UnstableTree$ are distinct points in $\cvclo$.  Both trees are
  mixing (Lemma \ref{lem:mixing}), but \cite[Proposition
  4.3]{H:CyclicSBoundary} provides that if two mixing trees in
  $\cvclo$ are equivalent (i.e., determine the same point in
  $\mathcal{X}^{\m}$), then each must collapse onto the other.  If
  there a collapse map from $T\to T'$, then $L(T)\subseteq L(T')$.  So
  if $\StableTree$ and $\UnstableTree$ were equivalent, then their
  dual laminations would be equal, a contradiction.
  
  We now argue that the limit set of $\langle\phi\rangle$ acting on
  $\CyclicS^{\m}$ consists of two points.  There is a minor
  complication arising from the fact that the folding path constructed
  in Section~\ref{sec:folding-seq} consisted entirely of trees in the
  boundary of outer space, but Theorem \ref{thm:horbez} applies only
  to sequences in the interior.  Indeed, recall from
  Section~\ref{sec:folding-seq} that $T$ denotes the universal cover
  of a relative train track map representing $\phi$ and that $T_0$ was
  obtained from $T$ by first collapsing the $\free$-translates of the
  $\mathcal{A}$-minimal subtree in $T$, then further collapsing
  according to a measure $\mu$.  Finally, recall (Proposition
  \ref{prop:T-Z-averse}) that the sequence
  $T_i=\PFevalue^{-i}T_0 \phi^i$ where $i\in\mathbb{N}$ converges to
  $\StableTree$, which is $\mathcal{Z}^{\m}$-averse.  Let
  $R_i=T\phi^i$ and let $R_\infty=\lim_{i\to\infty}R_i$.  For all
  $i\in\mathbb{N}$, $R_i$ collapses onto $T_i$, so $R_i$ and $T_i$ are
  compatible.  That compatibility passes to the limit follows from
  \cite[Corollary A.12]{GL:JSJ}, so $R_\infty$ is compatible with
  $\StableTree$ and is therefore $\mathcal{Z}$-averse.  Applying
  Theorem \ref{thm:horbez} to the sequence $\{R_i\}_{i\in\mathbb{N}}$,
  we conclude that the image sequence $\psi(R_i)$ converges to
  $[\StableTree]\in \mathcal{X}^{\m}$.  Finally, since the set of
  reducing splittings for a free simplicial $\free$-tree is bounded,
  if $S$ is any $\mathcal{Z}^{\m}$-splitting we have that $S\phi^i$
  converges to $[\StableTree]$, with a similar statement holding for
  iterates of $\phi^{-1}$.  Thus,
  $\Lambda_{\CyclicS}\langle\phi\rangle$ consists of exactly two
  points and $\phi$ therefore acts loxodromically on $\CyclicS^{\m}$.
   
  We now prove the converse: if $\phi$ acts loxodromically on
  $\CyclicS^{\m}$, then $\phi$ has a $\mathcal{Z}^{\m}$-filling
  lamination.  Indeed, if $\phi$ acts loxodromically on
  $\CyclicS^{\m}$, then $\phi$ necessarily acts loxodromically on
  $\FreeS$, and thus has a filling lamination $\lamination$.  If the
  lamination is not $\mathcal{Z}^{\m}$-filling, then Proposition
  \ref{prop:v-group} implies that a power of $\phi$ fixes a point in
  $\CyclicS^{\m}$, contradicting our assumption on $\phi$.  Thus,
  $\lamination$ is $\mathcal{Z}^{\m}$-filling.
\end{proof}

\section{Examples}
This section will provide several examples exhibiting the range of
behaviors of outer automorphisms acting on $\CyclicS$.  We begin with
an automorphism that acts loxodromically on $\CyclicS$.
\begin{example}[Loxodromic element]\label{Ex:lox1}
  Let $\phi$ be a rotationless automorphism with a CT representative
  $f\colon G\to G$ satisfying the following properties:
  \begin{itemize}
  \item $f$ has exactly two strata, each of which is EG and
    non-geometric
  \item the lamination corresponding to the top stratum of $f$ is
    filling
  \end{itemize}
  An explicit example satisfying these properties can be constructed
  using the sage-train-tracks package written by T.\ Coulbois
  \cite{C:Sage}.  The fact that the top lamination is filling
  guarantees that $\phi$ acts loxodromically on $\FreeS$.  As both
  strata are non-geometric, \cite[Fact 1.42(1a)]{HM:PartI} guarantees
  that $\phi$ does not fix the conjugacy class of any element of
  $\free$, and therefore cannot possibly fix a cyclic splitting.
  Theorem \ref{thm:classification} implies that $\phi$ acts
  loxodromically.
\end{example}

\begin{example}[Bounded orbit without fixed point]\label{Ex:bddorbit}
  By building on Example \ref{Ex:lox1} and \cite[Example
  4.2]{HM:FreeSplittingComplexII}, we can construct an automorphism
  $\psi$ which acts on $\CyclicS$ with bounded orbits but without a
  fixed point.  Let $\psi$ be a three stratum automorphism obtained
  from $f$ by creating a duplicate of $H_2$.  Explicitly, $\psi$ has a
  CT representative $f'\colon G'\to G'$ defined as follows. The graph
  $G'$ is obtained by taking two copies of $G$ and identifying them
  along $G_1$.  Each edge $E$ of $G'$ is naturally identified with an
  edge of $G$, and $f'(E)$ is defined via this identification.
  Moreover, the marking of $G$ naturally gives a marking of $G'$ (by a
  larger free group).  That $f'$ is a CT is evident from the fact that
  $f$ is a CT.

  There are three laminations in $\mathcal{L}(\psi)$, and it's evident
  that none are filling.  Since the top lamination in
  $\mathcal{L}(\phi)$ (where $\phi$ is as in Example~\ref{Ex:lox1}) is
  filling, we know that $\mathcal{L}(\psi)$ must fill.  Thus, $\psi$
  acts on $\FreeS$ with bounded orbits.  As before, \cite[Fact 1.42
  (1a)]{HM:PartI} implies that $\psi$ doesn't fix the conjugacy class
  of any element of $\free$: while each stratum may have an INP,
  $\rho_i$, none of these INPs can be closed loops, nor can they
  be concatenated to form a closed loop.  Thus, $\psi$ does not fix
  any one-edge cyclic splitting and therefore must act on $\CyclicS$
  with bounded orbits, but no fixed point.  See
  Figure~\ref{fig:EG-Splitting} for a pictorial representation of
  $\psi$.  The INPs $\rho_2$ and $\rho_3$ must each have at least one
  endpoint which is not in $H_1$.
  \begin{figure}[h]
    \centering{ \def\svgwidth{.6\linewidth} 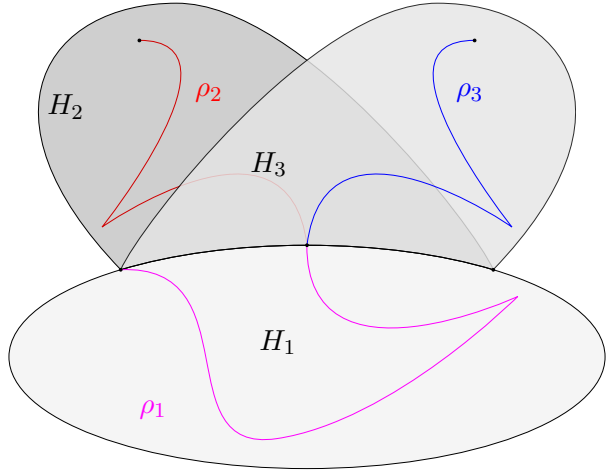
      \caption{A CT representative for the automorphism in Example
        \ref{Ex:bddorbit}, which acts with bounded orbits but no fixed
        point.}
      \label{fig:EG-Splitting}
    }
  \end{figure}
\end{example}

\begin{example}[Loxodromic element]\label{Ex:lox2}
  Consider the outer automorphism $\phi\colon F_4 \to F_4$ given by
  \begin{equation*}
    \phi(a) = ab, \phi(b) =bcab, \phi(c) = d, \phi(d) = cd.
  \end{equation*}
  In \cite{R:ReducingSystems}, it is shown that the stable tree for
  $\phi$ is indecomposable and hence $\mathcal{Z}$-averse.  Therefore
  $\phi$ acts loxodromically on $\CyclicS$.
\end{example}

\begin{example}[Fixed point]\label{Ex:fixedpt}
  Let $\Sigma_{2,1}$ be the surface of genus two with one puncture.
  Consider the free homotopy class of a simple separating curve which
  divides $\Sigma_{2,1}$ into two subsurfaces: a once punctured torus
  and a twice punctured torus.  Placing a pseudo-Anosov on each of
  these subsurfaces and taking the outer automorphism induced by this
  mapping class yields an element of $\Out(\free)$ that acts
  loxodromically on $\FreeS$, but fixes a point in $\CyclicS$.  A
  similar example using non-separating simple curve can be found in
  the proof of \cite[Proposition 3]{M:CyclicS}.
\end{example}

\section{Virtually Cyclic Centralizers}
In this section, we investigate centralizers of automorphisms acting
loxodromically on $\CyclicS$.  To do this, we use the machinery of
completely split train tracks, and the ``disintegration'' procedure of
\cite{FH:Abelian}, which takes a rotationless outer automorphism and
returns an abelian subgroup of $\Out(\free)$.  The main result is:

\theoremstyle{theorem} \newtheorem*{thm:centralizers}{Theorem
  \ref{thm:centralizers}}
\begin{thm:centralizers}
  An outer automorphism with a filling lamination has a virtually
  cyclic centralizer in $\Out(\free)$ if and only if the lamination is
  $\mathcal{Z}$-filling.
\end{thm:centralizers}

We begin with a terse review of disintegration for outer
automorphisms.

\subsection{Disintegration and rotationless abelian subgroups in
  $\Out(\free)$}

Given a mapping class $f$ in Thurston normal form, there is a
straightforward way of making a subgroup of the mapping class group,
called \emph{the disintegration of $f$}, by ``doing one piece at a
time.''  The subgroup is easily seen to be abelian as each pair of
generators can be realized as homeomorphisms with disjoint supports.
The process of disintegration in $\Out(\free)$ is analogous, but more
difficult.

The reader is warned that we will only review those ingredients from
\cite{FH:Abelian} that will be used directly; the reader is directed
there, specifically to \S 6, for complete details.  Given a
rotationless outer automorphism $\phi$, one can form an abelian
subgroup called $\mathcal{D}(\phi)$.  The process of disintegrating
$\phi$ begins by creating a finite graph, $B$, which records the
interactions between different strata in a CT representing $\phi$.  As
a first approximation, the components of $B$ correspond to generators
of $\mathcal{D}(\phi)$.  However, there may be additional relations
between strata that are unseen by $B$, so the number of components of
$B$ only gives an upper bound to the rank of $\mathcal{D}(\phi)$.

Let $f\colon G\to G$ be a CT representing the rotationless outer
automorphism $\phi$.  While the construction of $\mathcal{D}(\phi)$
does depend on $f$, using different representatives will produce
subgroups that are commensurable.

Let $E_i,E_j$ be distinct linear edges in $G$ with the same axis $w$
so that $f(E_i)=E_iw^{d_i}$ and $f(E_j)=E_jw^{d_j}$ for integers
$d_i\neq d_j$.  Recall that if $d_i,d_j>0$, then we a path of the form
$E_iw^*\overline{E}_j$ called an \emph{exceptional path}.  In the same
scenario, if $d_i$ and $d_j$ have different signs, we call such a path
a \emph{quasi-exceptional path}.  It would be instructive for the
reader to compute the $f$-image of some exceptional and
quasi-exceptional paths.  We will need to consider a weakening of the
complete splitting of paths and circuits in $f$.  The
\emph{quasi-exceptional splitting} of a completely split path or
circuit $\sigma$ is the coarsening of the complete splitting obtained
by considering each quasi-exceptional subpath to be a single element.

\begin{definition}
  Define a finite directed graph $B$ as follows.  There is one vertex
  $v_i^B$ for each nonfixed irreducible stratum $H_i$.  If $H_i$ is
  NEG, then a $v_i^B$-path is defined as the unique edge in $H_i$; if
  $H_i$ is EG, then a $v_i^B$-path is either an edge in $H_i$ or a
  taken connecting path in a zero stratum contained in $H_i^z$.  There
  is a directed edge from $v_i^B$ to $v_j^B$ if there exists a
  $v_i^B$-path $\kappa_i$ such that some term in the QE-splitting of
  $f_{\#}(\kappa_i)$ is an edge in $H_j$.  The components of $B$ are
  labeled $B_1,\ldots, B_K$.  For each $B_s$, define $X_s$ to be the
  minimal subgraph of $G$ that contains $H_i$ for each NEG stratum
  with $v_i^B\in B_s$ and contains $H_i^z$ for each EG stratum with
  $v_i^B\in B_s$.  We say that $X_1,\ldots,X_K$ are the \emph{almost
    invariant subgraphs} associated to $f\colon G\to G$.
\end{definition}

The reader should note that the number of components of $B$ is left
unchanged if an iterate of $f_\#$ is used in the definition, rather
than $f_\#$ itself.  In the sequel, we will frequently make statements
about $B$ using an iterate of $f_\#$.

For each $K$-tuple $\vec{a}=(a_1,\ldots,a_K)$ of non-negative
integers, define
\begin{equation*}
  f_{\vec{a}}(E)=
  \begin{cases}
    f_{\#}^{a_i}(E)	&\text{if }E\in X_i\\
    E &\text{if }E\text{ is fixed by }f
  \end{cases}
\end{equation*}
It turns out that $f_{\vec{a}}$ is always a homotopy equivalence of
$G$ \cite[Lemma 6.7]{FH:Abelian}, but in general
$\langle f_{\vec{a}}\mid \vec{a}\text{ is a non-negative
  tuple}\rangle$ is not abelian.  To obtain an abelian subgroup, one
has to pass to a certain subset of tuples which take into account
interactions between the almost invariant subgraphs that are unseen by
$B$. The reader is referred to \cite[Example 6.9]{FH:Abelian} for an
example.

\begin{definition}\label{def:admissible}
  A $K$-tuple $(a_1,\ldots,a_K)$ is called \emph{admissible} if for
  all axes $\mu$, whenever
  \begin{itemize}
  \item $X_s$ contains a linear edge $E_i$ with axis $\mu$ and
    exponent $d_i$,
  \item $X_t$ contains a linear edge $E_j$ with axis $\mu$ and
    exponent $d_j$,
  \item there is a vertex $v^B$ of $B$ and a $v^B$-path
    $\kappa\subseteq X_r$ such that some element in the
    quasi-exceptional family $E_i\overline{E}_j$ is a term in the
    QE-splitting of $f_\#(\kappa)$,
  \end{itemize}
  then $a_r(d_i-d_j)=a_sd_i-a_td_j$.
\end{definition}

The disintegration of $\phi$ is then defined as
\begin{equation*}
  \mathcal{D}(\phi)=\langle f_{\vec{a}}\mid \vec{a}\text{ is admissible}\rangle,
\end{equation*}

which is abelian by \cite[Corollary 6.16]{FH:Abelian}.

We now recall some useful facts concerning abelian subgroups of
$\Out(\free)$, which were studied in \cite{FH:Abelian}.  If an abelian
subgroup $H$ is generated by rotationless automorphisms, then all
elements of $H$ are rotationless \cite[Corollary 3.13]{FH:Abelian}.
In this case, $H$ is said to be rotationless.  Rotationless abelian
subgroups of $\Out(\free)$ have finitely many attracting laminations
(\cite[Lemma 4.4]{FH:Abelian}), i.e., if $H$ is abelian and
$\mathcal{L}(H):=\bigcup_{\phi\in H}\mathcal{L}(\phi)$, then
$|\mathcal{L}(H)|<\infty$.

In \cite{FH:Abelian}, the authors associate to each rotationless
abelian subgroup of $\Out(\free)$ a finite collection of (nontrivial)
homomorphisms to $\mathbb{Z}$.  Combining these, one obtains a
homomorphism $\Omega\colon H\to \mathbb{Z}^N$ that is injective
\cite[Lemma 4.6]{FH:Abelian}.  An element $\psi\in H$ is said to be
\emph{generic} if all coordinates of $\Omega(\psi)$ are nonzero.  For
the purposes of this section, we require only two facts concerning
$\Omega$.  First, some of the coordinates of $\Omega$ correspond to
elements in the finite set $\mathcal{L}(H)$ (there are other
coordinates, which we will not need).  Second is the fact that the
coordinate of $\Omega(\psi)$ corresponding to
$\Lambda\in\mathcal{L}(H)$ is positive if and only if
$\Lambda\in\mathcal{L}(\psi)$.

\subsection{From disintegrations to centralizers}
In this subsection, we explain how to deduce Theorem
\ref{thm:centralizers} from the following proposition concerning the
disintegration of elements acting loxodromically on $\CyclicS$.  The
proof of Proposition~\ref{prop:D-cyclic} is postponed until the next
subsection.

\begin{proposition}\label{prop:D-cyclic}
  If $\phi$ is rotationless and has a $\mathcal{Z}$-filling
  lamination, then $\mathcal{D}(\phi)$ is virtually cyclic.
\end{proposition}

\begin{proof}[Proof of Theorem \ref{thm:centralizers}]
  Suppose $\psi\in C(\phi)$ has infinite order and assume that
  $\langle \phi,\psi\rangle\simeq \mathbb{Z}^2$.  If no such element
  exists, then $C(\phi)$ is virtually cyclic, as there is a bound on
  the order of a finite subgroup of $\Out(\free)$ \cite{C:Finite}.
  Now let $H_R$ be the finite index subgroup of
  $\langle \phi,\psi\rangle$ consisting of rotationless elements
  \cite[Corollary 3.14]{FH:Abelian} and let $\psi'$ be a generic
  element of this subgroup.  If the coordinate of $\Omega(\psi')$
  corresponding to the $\mathcal{Z}$-filling lamination $\lamination$
  is negative, then replace $\psi'$ by $(\psi')^{-1}$, which is also
  generic.  Since $\lamination\in \mathcal{L}(\psi')$ is
  $\mathcal{Z}$-filling, Theorem~\ref{thm:loxodromics} implies that
  $\psi'$ acts loxodromically on $\CyclicS$.  Since $\psi'$ is generic
  in $H_R$, \cite[Theorem 7.2]{FH:Abelian} says that
  $\mathcal{D}(\psi')\cap \langle \phi,\psi\rangle$ has finite index
  in $\langle\phi,\psi\rangle$.  This contradicts Proposition
  \ref{prop:D-cyclic}, which says that the disintigration of $\psi'$
  is virtually cyclic.
\end{proof}

\subsection{The proof of Proposition~\ref{prop:D-cyclic}}
The idea of the proof is as follows.  We noted above that the number
of components in $B$ only gives an upper bound to the rank of
$\mathcal{D}(\phi)$; it may happen that there are interactions between
the strata of $f$ that are unseen by $B$ (Definition
\ref{def:admissible}).  We will obtain precise information about the
structure of $B$; it consists of one main component ($B_1$), and
several components consisting of a single point ($B_2,\ldots,B_K$).
We will then show that the admissibility condition provides
sufficiently many constraints so that choosing $a_1$ determines
$a_2,\ldots,a_K$.  Thus, the set of admissible tuples consists of a
line in $\mathbb{Z}^K$.

Let $f\colon G\to G$ be a CT representing $\phi$ with filtration
$\fltr$.  Let $\lamination\in\mathcal{L}(\phi)$ be
$\mathcal{Z}$-filling and let $\ell\in\lamination$ be a generic leaf.
As $\lamination$ is filling, the corresponding EG stratum is
necessarily the top stratum, $H_M$.  We will understand the graph $B$
by studying the realization of $\ell$ in $G$.  The results of \cite[\S
3.1]{BFH:Tits}, together with Lemma 4.25 of
\cite{FH:RecognitionTheorem} give that the realization of $\ell$ in
$G$ is completely split, and this splitting is unique.  Thus, we may
consider the QE-splitting of $\ell$.

We begin with a lemma that allows the structure of INPs and
quasi-exceptional paths to be understood inductively.

\begin{lemma}\label{lem:INPQE-ind}
  Let $H_r$ be a non-fixed irreducible stratum and let $\rho$ be a
  path of height $s\geq r$ which is either an INP or a
  quasi-exceptional path.  Assume further that $\rho$ intersects $H_r$
  non-trivially.  Then one of the following holds:
  \begin{itemize}
  \item $H_r$ and $H_s$ are NEG linear strata with the same axis, each
    consisting of a single edge $E_r$ (resp.\ $E_s$), and
    $\rho=E_sw^k\overline{E}_r$, for some $k \in \mathbb{Z}$, where
    $w$ is a closed, root-free Nielsen path of height $<s$.
  \item $\rho$ can be written as a concatenation
    $\rho=\beta_0\rho_1\beta_1\rho_2\beta_2\ldots\rho_j\beta_j$, where
    each $\rho_i$ is an INP of height $r$ and each $\beta_i$ is a path
    contained in $G-\inter(H_r)$ (some of the $\beta_i$'s may be
    trivial).
  \end{itemize}
\end{lemma}

\begin{proof}
  The proof proceeds by strong induction on the height $s$ of the path
  $\rho$.  In the base case, $s=r$, and $\rho$ is either an INP of
  height $r$ or a quasi-exceptional path of the form described.  The
  inductive step breaks into cases according whether $H_s$ is an EG
  stratum, or an NEG stratum.

  If $H_s$ is an EG stratum, then $\rho$ must be an INP, as there are
  no exceptional paths of EG height.  In this case, \cite[Lemma 4.24
  (2)]{FH:RecognitionTheorem} provides a decomposition of $\rho$ into
  subpaths of height $s$ and maximal subpaths of height $<s$, and each
  of the subpaths of height $<s$ is a Nielsen path.  The inductive
  hypothesis then guarantees that each of these Nielsen paths has the
  desired form.  By breaking apart and combining these terms
  appropriately, we conclude that $\rho$ does as well.

  Suppose now that $H_s$ is an NEG stratum and let $E_s$ be the unique
  edge in $H_s$.  Using (NEG Nielsen Paths), we see that $E_s$ must be
  a linear edge, and therefore that $\rho$ is either
  $E_s w^k\overline{E_s}$ or $E_s w^k\overline{E'}$, where $E'$ is
  another linear edge with the same axis and $w$ is a closed root free
  Nielsen path of height $<s$.  If $H_r$ is NEG linear, and $E'=E_r$,
  then the first conclusion holds.  Otherwise, we may apply the
  inductive hypothesis to $w$ to obtain a decomposition as desired.
  This completes the proof.
\end{proof}

We now begin our study of the graph $B$.  We call the component of $B$
containing $v_M^B$, the vertex corresponding to the topmost stratum of
$f$, the \emph{main component}.

\begin{lemma}\label{lem:NEGNL-B}
  All nonlinear NEG strata are in the main component of $B$.
\end{lemma}
\begin{proof}
  Let $H_r$ be a nonlinear NEG stratum, with single edge $E_r$.  It is
  enough to show that the single edge $E_r$ occurs as a term in the
  QE-splitting of $\ell$ (henceforth, we will say that $E_r$ is a
  \emph{QE-splitting unit} in $\ell$), as this implies that there is
  an edge in $B$ connecting $v_M^B$ to $v_r^B$.  As $\ell$ is filling,
  we know that its realization in $G$ must cross $E_r$.  If the
  corresponding QE-splitting unit of $\ell$ is the single edge $E$,
  then we are done.  The only other possibility is that the
  QE-splitting unit is an INP or a quasi-exceptional path of some
  height $s\geq r$.  An application of Lemma \ref{lem:INPQE-ind} shows
  that this is impossible, as it would imply the existence of an INP
  of height $r$ or a quasi-exceptional path of the form
  $E_rw^*\overline{E}'$, contradicting (NEG Nielsen Paths).
\end{proof}

\begin{lemma}\label{lem:EG-B}
  All EG strata are in the main component of $B$.
\end{lemma}
\begin{proof}
  Let $H_r$ be an EG stratum.  As before, it is enough to show that
  some (every) edge of $H_r$ occurs as a QE-splitting unit of $\ell$.
  There are three types of QE-splitting units that can cross $H_r$: a
  single edge in $H_r$, an INP of height $\geq r$, or a
  quasi-exceptional path.  In the first case, we are done, so suppose
  that every time $\ell$ crosses $H_r$, the corresponding QE-splitting
  unit is an INP or a quasi-exceptional path.  We now argue that this
  situation leads to a contradiction.

  We may write $\ell$ as a concatenation
  $\ell=\ldots\gamma_1\sigma_1\gamma_2\sigma_2\ldots$ where each
  $\sigma_i$ is a QE-splitting unit of $\ell$ which intersects
  $\inter(H_r)$, and each $\gamma_i$ is a maximal concatenation of
  QE-splitting units of $\ell$ which do not intersect $\inter(H_r)$
  (some $\gamma_i$'s may be trivial).  By assumption, each $\sigma_i$
  is an INP or a QEP.  Applying Lemma \ref{lem:INPQE-ind} to each of
  the $\sigma_i$'s, then combining and breaking apart the terms
  appropriately, we see that $\ell$ can be written as a concatenation
  $\ell=\ldots\gamma_1\rho_1\gamma_2\rho_2\ldots$ where each $\rho_i$
  is the unique INP of height $r$ or its inverse.  Call this INP
  $\rho$.

  We will now use the information we have about $\ell$ to find a
  $\mathcal{Z}$-splitting in which $\ell$ is carried by a vertex
  group.  The existence of such a splitting will contradict our
  assumption that $\ell$ is a generic leaf of the
  $\mathcal{Z}$-filling lamination $\lamination$.
  
  We now modify $G$ to produce a 2-complex, $G''$, whose
  fundamental group is identified with $\free$.  First assume $H_r$ is
  non-geometric, so that $\rho$ has distinct endpoints, $v_0$ and
  $v_1$.  Let $G'$ be the graph obtained from $G$ by replacing each
  vertex $v_i$ for $i\in\{0,1\}$ with two vertices, $v_i^u$ and
  $v_i^d$ ($u$ and $d$ stand for ``up'' and ``down''), which are to be
  connected by an edge $E_i$.  For each edge $E$ of $G$ incident to
  $v_i$, connect it in $G'$ to the new vertices as follows: if
  $E\in H_r$, then $E$ is connected to $v_i^d$, and if $E\notin H_r$,
  then $E$ is connected to $v_i^u$.  $G'$ deformation retracts onto
  $G$ by collapsing the new edges, and this retraction identifies
  $\pi_1(G')$ with $\free$ via the marking of $G$.  Let
  $R=[0,1]\times [0,1]$ be a rectangle and define $G''$ by gluing
  $\{i\} \times [0,1]$ homeomorphically onto $E_i$ for $i\in\{0,1\}$,
  then gluing $ [0,1] \times \{0\}$ homeomorphically to the INP
  $\rho$.  As only three sides of the rectangle have been glued, $G''$
  deformation retracts onto $G'$, and its fundamental group is again
  identified with $\free$.

\begin{figure}[h]
  \centering{ \def\svgwidth{.5\linewidth} 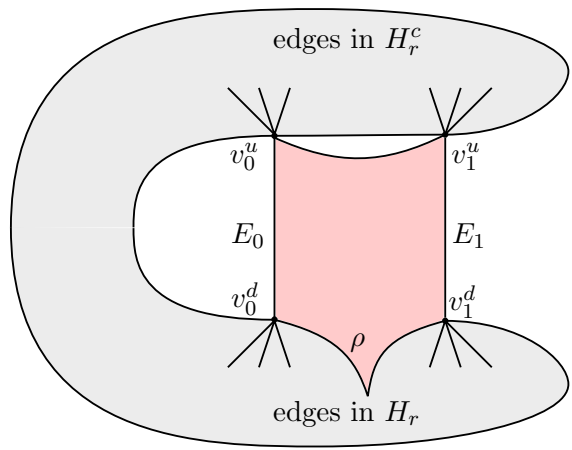
    \caption{$G''$ when $H_r$ is a non-geometric EG stratum}
    \label{fig:EG-Splitting}
  }
\end{figure}

The construction of $G''$ differs only slightly if $H_r$ is geometric.
In this case, $\rho$ is a closed loop based at $v_0$ and we blow up
$v_0$ to two vertices, $v_0^u$ and $v_0^d$, that are connected by an
edge $E_0$.  Instead of gluing in a rectangle, we glue in a cylinder
$R=S^1 \times [0,1]$; $\{p\} \times [0,1]$ is glued homeomorphically
to $E_0$ where $p$ is a point in $S^1$, and $S^1 \times \{0\}$ is
glued homeomorphically to $\rho$.

Recall that in $G$, the leaf $\ell$ can be written as a concatenation
$\ell=\ldots\gamma_1\rho_1\gamma_2\rho_2\ldots$ where each $\rho_i$ is
either $\rho$ or $\overline{\rho}$.  Thus we can realize $\ell$ in
$G'$ as $\ell=\ldots\gamma_1\rho'_1\gamma_2\rho'_2\ldots$ where each
$\rho'_i$ is either $E_0\rho\overline{E}_1$ or
$E_1\overline{\rho}\overline{E}_0$.  In $G''$, each $\rho_i'$ is
homotopic rel endpoints to a path that travels along the top of $R$,
rather than down-across-and-up.  Thus, after performing a (proper!)
homotopy to the image of $\ell$, we can arrange that it never
intersects the interior of $R$, nor the vertical sides of $R$.
Cutting $R$ along its centerline yields a $\mathcal{Z}$-splitting $S$
of $\free$, and $\ell$ is carried by a vertex group of this splitting.
If $H_r$ is non-geometric, then $S$ is a free splitting and if $H_r$
is geometric, then $S$ is a cyclic splitting. In either case, so long
as $S$ is non-trivial, we have contradicted our assumption that the
lamination is $\mathcal{Z}$-filling.
\end{proof}

\begin{claim}\label{claim:non-triv}
  The splitting $S$ is non-trivial.
\end{claim}

\begin{proof}[Proof of Claim \ref{claim:non-triv}]
  We first handle the case that $H_r$ is geometric.  We have described
  a one-edge cyclic splitting $S$ which was obtained as follows: cut
  $G'$ along the edge $E_0$, that is, collapse $G' - E_0$ to get a
  free splitting of $\free$, then perform the edge fold corresponding
  to $\langle w\rangle$ (see Section~\ref{subsec:CyclicS} for
  definition), where $w$ is the conjugacy class of the INP $\rho$.  If
  $G'-E_0$ is connected, then the free splitting is an HNN extension,
  and there is no danger of $S$ being trivial as $\rk(\free)\geq 3$.
  On the other hand, if $G'-E_0$ is disconnected, then let ${G^d}'$
  and ${G^u}'$ be the components of $G'-E_0$ containing $v_0^d$ and
  $v_0^u$ respectively.  The free splitting which is folded to get $S$
  is precisely $\pi_1({G^d}')\ast \pi_1({G^u}')$.  In this case,
  ${G^d}'$ is necessarily a component of $G_r$ and \cite[Proposition
  2.20 (2)]{FH:RecognitionTheorem} together with (Filtration) imply
  that this component is a core graph.  As $H_r$ is EG, the rank of
  $\pi_1({G^d}')$ is at least two and the splitting $S$ is therefore
  non-trivial.  To see that $\rk(\pi_1({G^u}'))\geq 1$, we need only
  recall that $\ell$ is not periodic and is carried by
  $\pi_1({G^c}')\ast\langle w\rangle$.

  In the case that $H_r$ is non-geometric, the splitting obtained
  above is a free splitting.  If $G'-\{E_0,E_1\}$ is connected, then
  the free splitting is an HNN extension, and as before $S$ is
  non-trivial. If $G'-\{E_0, E_1\}$ is disconnected, then the
  component containing $v_0^d$ (and by necessity $v_1^d$), denoted
  ${G^d}'$, corresponds to a vertex group of $S$. By the same
  reasoning as in the previous case, we get that $\pi_1({G^d}')$ is
  non-trivial.  As before, the other vertex group of $S$ carries the
  leaf $\ell$ and hence $S$ is a non-trivial free splitting.
\end{proof}

\begin{remark}
  We would like the reader to note that the above proof actually gives
  restrictions on the way two EG strata in a CT can interact.  For
  example, suppose that $\phi$ is represented by a CT,
  $f\colon G\to G$, with exactly two strata, both of which are EG.
  Assume further that $H_1$ is non-geometric and has an INP.  A priori,
  there are three ways that $H_2$ can interact with $H_1$:
  \begin{enumerate*}
  \item there is some edge $E$ in $H_2$ such that $f_{\#}(E)$ contains
    an edge from $H_1$ as a splitting unit,
  \item the $f_\#$ image of each edge in $H_2$ is entirely contained
    in $H_2$, or
  \item whenever $E$ is an edge from $H_2$ and $f_\#(E)$ crosses
    $H_1$, the corresponding splitting unit is the INP of height 1.
  \end{enumerate*}
  In the first case, $\Lambda_2\supset\Lambda_1$.  In the second case,
  we may think of the strata as being side-by-side, rather than $H_2$
  being stacked on top of $H_1$.  The proof of Lemma \ref{lem:EG-B}
  implies that the third possibility never happens.  Indeed, the proof
  provides a free splitting which is $\phi$-invariant and the vertex
  groups of this splitting form a free factor system which lies
  strictly between the free factor systems $\pi_1(G_1)$ and
  $\pi_1(G_2)$.  This contradicts (Filtration) in the definition of a
  CT, which states that the filtration $\fltr$ must be reduced.
\end{remark}

Before addressing the NEG linear strata and concluding the proof of
Proposition~\ref{prop:D-cyclic}, we present a final lemma concerning
the structure of $B$.

\begin{lemma} Assume $H_r$ is a linear NEG stratum consisting of an
  edge $E_r$.  If $v_r^B$ is not in the main component of $B$, then
  the component of $B$ containing $v_r^B$ is a single point.
\end{lemma}

\begin{proof}
  This follows directly from the definition of $B$, together with
  Lemmas \ref{lem:NEGNL-B} and \ref{lem:EG-B}.  If $H_r$ is a linear
  NEG stratum, then the definition of $B$ implies that $v_r^B$ has no
  outgoing edges.  For any edge in $B$ whose terminal vertex is
  $v_r^B$, its initial vertex necessarily corresponds to a non-linear
  NEG stratum or an EG stratum, and hence is in the main component of
  $B$.
\end{proof}

When dealing with an NEG linear stratum, we would like to carry out a
similar strategy to the EG case: blow up the terminal vertex, $v_0$,
to an edge and glue in a cylinder, thereby producing a cyclic
splitting in which $\ell$ is carried by a vertex group.  The main
difficulty in implementing this comes from other linear edges with the
same axis; for each such edge, one has to decide whether to glue it in
$G'$ to $v_0^d$ or $v_0^u$.

Let $\mu$ be an axis with corresponding unoriented root-free conjugacy
class $w$.  Let $\mathcal{E}_\mu$ be the set of linear edges in $G$
with axis $\mu$.  Define a relation on $\mathcal{E}_\mu$ by declaring
$E\sim_R E'$ if the quasi-exceptional path $Ew^*\overline{E}'$ is a
QE-splitting unit in $\ell$ or if both $E$ and $E'$ are QE-splitting
units in $\ell$.  Then let $\sim$ be the equivalence relation
generated by $\sim_R$.  Note that all edges in $\mathcal{E}_\mu$ which
occur as QE-splitting units in $\ell$ are in the same equivalence
class.

As mentioned above, the difficulty in adapting the strategy used for
EG stratum to the present situation lies in deciding where to glue
edges (top or bottom) in $G'$.  The existence of multiple classes in
the equivalence relation $\sim$ will provide instructions for how to
glue edges from $\mathcal{E}_\mu$ in $G'$ so that the leaf never
crosses the cylinder in $G''$.

\begin{lemma}
  There is only one equivalence class of $\sim$.  Moreover, at least
  one edge in $\mathcal{E}_\mu$ occurs as a term in the QE-splitting
  of $\ell$.
\end{lemma}
\begin{proof}
  Suppose for a contradiction that there is more than one equivalence
  class of $\sim$ and $[E]$ be an equivalence class for which no edge
  in $[E]$ is a QE-splitting unit in $\ell$.  Now build $G'$ as in the
  proof of Lemma \ref{lem:EG-B}.  Let $v_0$ be the terminal vertex of
  the edges in $\mathcal{E}_\mu$ (they all have the same terminal
  vertex), and define $G'$ by blowing up $v_0$ into two vertices,
  $v_0^u$ and $v_0^d$, which are connected by an edge $E_0$.  The
  terminal vertex of each edge of $[E]$ is to be glued in $G'$ to
  $v_0^u$, while all other edges in $G$ that are incident to $v_0$ are
  glued to $v_0^d$.  Define $G''$ as before, gluing the bottom of a
  cylinder $R$ along the closed loop $w$, and gluing the vertical
  interval above $v_0$ homeomorphically to the edge $E_0$.

  The definition of $\sim$ guarantees that $\ell$ is carried by a
  vertex group of the cyclic splitting determined by cutting along the
  centerline of $R$.  Indeed, whenever $\ell$ crosses an edge from
  $[E]$, the corresponding QE-splitting unit is either an INP or a
  quasi-exceptional path $E'w^*\overline{E}''$, where $E',E''\in [E]$.
  Repeatedly applying Lemma \ref{lem:INPQE-ind} to each of these
  terms, then rearranging and combining terms appropriately, we see
  that $\ell$ can be written in $G$ as a concatenation
  $\ell=\ldots\gamma_1\rho_1\gamma_2\rho_2\ldots$ where each $\rho_i$
  is either $E'w^*\overline{E'}$ or $E'w^*\overline{E''}$ with
  $E',E''\in [E]$.  Thus we can realize $\ell$ in $G'$ as
  $\ell=\ldots\gamma_1\rho'_1\gamma_2\rho'_2\ldots$ where each
  $\rho'_i$ is $E'E_0w^*\overline{E}_0\overline{E'}$ or
  $E'E_0w^*\overline{E}_0\overline{E''}$.  In $G''$, each $\rho_i'$ is
  homotopic rel endpoints to a path that travels along the top of $R$,
  rather than down-across-and-up.  Thus, we have again produced a
  cyclic splitting in which $\ell$ is carried by a vertex group.

  We now argue that the splitting is non-trivial.  There is a free
  splitting $S$ which comes from cutting the edge $E_0$ in $G'$, which
  cannot be a self loop.  The cyclic splitting of interest $S'$ is
  obtained from $S$ by performing the edge fold corresponding to $w$.  If
  $G'-E_0$ is connected, then $S'$ is an HNN extension with edge group
  $\langle [w]\rangle$.  As $\rk(\free)\geq 3$, the vertex group has
  rank at least two and we are done.  Now suppose $E_0$ is separating
  so that $G'-E_0$ consists of two components.  Let ${G'}^u$ be the
  component containing the vertex $v_0^u$ and let ${G'}^d$ be the
  other component.  The vertex groups of the splitting $S'$ are
  $\pi_1(G^d)$ and $\pi_1(G^u)\ast \langle [w]\rangle$.  The fact that
  $v$ is a principal vertex guarantees that
  $\pi_1(G^d)\not\cong \mathbb{Z}$, and the fact that $G$ is a finite
  graph without valence one vertices ensures that $\pi_1(G^u)$ is
  non-trivial.

  The proof of the second statement is exactly the same as that of the
  first.
\end{proof}

Finally, we finish the proof of Proposition~\ref{prop:D-cyclic}.  As
before, $B_1$ is the main component of $B$, with corresponding almost
invariant subgraph $X_1$.  All other components $B_2,\ldots, B_K$ are
single points, and each almost invariant subgraph $X_i$ consist of a
single linear edge.  Let $(a_1,\ldots,a_K)$ be a $K$-tuple and suppose
that $a_1$ has been chosen.  We claim that imposing the admissibility
condition determines all other $a_i$'s.

Suppose first that $E_i,E_j$ are linear edges with the same axis,
$\mu$, such that $E_i\in X_1$, $E_j\in X_k$, and $E_i\sim_R E_j$.  Let
$d_i$ and $d_j$ be the exponents of $E_i$ and $E_j$ respectively.
Applying the definition of admissibility with $s=r=1$, $t=k$, and
$\kappa$ a $v^B$ path such that $f_\#(\kappa)$ contains a
quasi-exceptional path of the form $E_iw^*\overline{E}_j$ in its
QE-splitting (such a $\kappa$ must exist as a quasi-exceptional path
of this type occurs in the QE-splitting of $\ell$), we obtain the
relation $a_1(d_i-d_j)=a_1d_i-a_kd_j$.  Thus $a_k$ is determined by
$a_1$.

Now suppose $E_i$ and $E_j$ are as above, but rather than being
related by $\sim_R$, we only have that $E_i\sim E_j$.  There is a
finite chain of $\sim_R$-relations to get from $E_i$ to $E_j$.  At
each stage in this chain, the definition of admissibility (applied
with $r=1$ and $\kappa$ chosen appropriately) will impose a relation
that determines the next coordinate from the previous ones.
Ultimately, this determines $a_k$.

We have thus shown that an admissible tuple is completely determined
by choosing $a_1$, and therefore that the set of admissible tuples
forms a line in $\mathbb{Z}^K$.  Therefore $\mathcal{D}(\phi)$ is
virtually cyclic.

\subsection{A Converse}
\begin{proposition}\label{prop:converse}
  If $\phi$ has a filling lamination which is not
  $\mathcal{Z}$-filling, then the centralizer of some power of $\phi$ in
  $\Out(\free)$ is not virtually cyclic.
\end{proposition}
\begin{proof}
  Since $\phi$ has filling lamination which is not
  $\mathcal{Z}$-filling, it follows by Proposition~\ref{prop:v-group}
  that for some $k$, $\phi^k$ fixes a one-edge cyclic splitting $S$.

  Suppose $S/\free$ is a free product with amalgamation with vertex
  stabilizers $\la A, w \ra$ and $B$ and edge group
  $\la w \ra \subset B$. Consider the Dehn twist $D_w$ given by $S$ as
  follows: $D_w$ acts as identity on $B$ and conjugation by $w$ on
  $A$. The automorphism $D_w$ has infinite order. We claim that $D_w$
  and $\phi^k$ commute. Indeed, consider a generating set
  $\{a_1, \ldots, a_k, b_1, \ldots b_m\}$ for $\free$ such that the
  $a_i$'s generate $A$ and the $b_i$'s generate $B$.  Choose a
  representative $\Phi$ of $\phi$ such that $\Phi^k(B) = B$ and
  $\Phi^k(\la A, w \ra) = \la A, w \ra^b$ for some element $b \in B$.
  Since $D_w$ is identity on $B$ and $\Phi^k(B) = B$, we have
  $\Phi^k(D_w(b_i)) = D_w(\Phi^k(b_i))$ for all generators $b_i$.  Since
  $D_w(a_i) = wa_i\overline{w}$, $\Phi^k(w) = w$ and
  $\Phi^k(\la A, w \ra) = \la A, w \ra^b$, we have
  $D_w(\Phi^k(a_i)) = \Phi^k(D_w(a_i))$ for all generators $a_i$.  Thus
  $D_w$ and $\phi^k$ commute.

  We now address the case that $S/ \free$ is an HNN extension.  Assume
  $S/\free$ has stable letter $t$, edge group $\la w \ra$ and vertex
  group $\la A, \overline{t}wt \ra$.  Since the cyclic splitting $S$ is obtained from a free HNN extension, with vertex group $A$ and stabe letter $t$, by an edge fold, we have that a basis of $\free$ is given
  by $\{a_1, a_2, \ldots, a_k, t \}$, where the $a_i$'s generate
  $A$. Consider the Dehn twist $D_w$ determined by $S$ such that $D_w$
  is identity on $A$ and sends $t$ to $wt$.  The automorphism $D_w$ has
  infinite order.  Choose a representative $\Phi$ of $\phi$ such that  $\la A, \overline{t}wt \ra$ is
  $\Phi^k$-invariant. Then for every generator $a_i$, $\Phi^k(a_i)$ is a word
  in the $a_i$s and $\overline{t}wt$.  Since $D_w$ is identity on $A$
  and fixes $\overline{t}wt$, we get
  $\Phi^k(D_w(a_i)) = D_w(\Phi^k(a_i))$.  Again, since
  $\la A, \overline{t}wt \ra$ is $\Phi^k$-invariant, $\Phi^k(t)$ is equal
  to $w^m t\alpha$, where $\alpha$ is some word in
  $\la A, \overline{t}wt\ra$ and $m \in \mathbb{Z}$. On one hand,
  $\Phi^k(D_w(t)) = \Phi^k(wt) = \Phi^k(w) \Phi^k(t) = w w^m t \alpha$ and on
  the other hand,
  $D_w(\Phi^k(t)) = D_w(w^m t \alpha) = w^m D_w(t) D_w(\alpha) = w^m w t
  \alpha$. Thus $D_w$ and $\phi^k$ commute.

  Thus when $\phi^k$ fixes a cyclic splitting, then an infinite order
  element other than a power of $\phi^k$ exists in the centralizer of
  $\phi^k$.
\end{proof}

\bibliographystyle{alpha}
\bibliography{BibFile}
\end{document}